\documentclass[a4paper,leqno]{amsart}

\title[Propagation of Gabor singularities]{Propagation of Gabor singularities for Schr\"odinger equations with quadratic Hamiltonians}

\author[K. Pravda--Starov]{Karel Pravda--Starov}

\address{IRMAR, CNRS UMR 6625, Universit\'e de Rennes 1, Campus de Beaulieu, 263 avenue du G\'en\'eral Leclerc, CS 74205,
35042 Rennes cedex, France}

\email{karel.pravda-starov@univ-rennes1.fr}

\author[L. Rodino]{Luigi Rodino}

\address{Department of Mathematics, University of Turin, Via Carlo Alberto 10, 10123 Torino (TO), Italy}

\email{luigi.rodino@unito.it}

\author[P. Wahlberg]{Patrik Wahlberg}

\address{Department of Mathematics, Linn{\ae}us University, SE--35195 V{\"a}xj{\"o}, Sweden}

\email{patrik.wahlberg@lnu.se}

\usepackage{latexsym}
\usepackage{amsmath}
\usepackage{amssymb}
\usepackage{amsthm}
\usepackage{amsfonts}
\usepackage{mathrsfs}
\usepackage{calc}
\usepackage{cite}
\usepackage{color}

\numberwithin{equation}{section}          

\newtheorem{thm}{Theorem}
\numberwithin{thm}{section}

\newcommand{\rubrik}{}
\newtheorem{prop}[thm]{Proposition}
\newtheorem{cor}[thm]{Corollary}
\newtheorem{lem}[thm]{Lemma}

\theoremstyle{definition}

\newtheorem{defn}[thm]{Definition}
\newtheorem{example}[thm]{Example}

\theoremstyle{remark}

\newtheorem{rem}[thm]{Remark}              

\newcommand{\ro}{\mathbb R}
\newcommand{\no}{\mathbb N}
\newcommand{\rr}[1]{\mathbb R^{#1}}
\newcommand{\nn}[1]{\mathbb N^{#1}}

\newcommand{\co}{\mathbb C}
\newcommand{\cc}[1]{\mathbb C^{#1}}

\newcommand{\Ker}{\operatorname{Ker}}
\newcommand{\Ran}{\operatorname{Ran}}

\newcommand{\ep}{\varepsilon}
\newcommand{\fy}{\varphi}

\newcommand{\supp}{\operatorname{supp}}

\newcommand{\eabs}[1]{\langle #1\rangle}

\newcommand{\Sp}{\operatorname{Sp}}
\newcommand{\ssp}{\operatorname{sp}}
\newcommand{\Mp}{\operatorname{Mp}}
\newcommand{\GL}{\operatorname{GL}}
\newcommand{\charac}{\operatorname{char}}
\newcommand{\conesupp}{\operatorname{conesupp}}

\newcommand{\cL}{\mathscr{L}}
\newcommand{\cS}{\mathscr{S}}
\newcommand{\cD}{\mathscr{D}}
\newcommand{\cF}{\mathscr{F}}
\newcommand{\cK}{\mathscr{K}}

\newcommand{\wh}{\widehat}
\newcommand{\re}{{\rm Re} \, }
\newcommand{\im}{{\rm Im} \, }
\newcommand{\J}{\mathcal{J}}
\def\la{\langle}
\def\ra{\rangle}

\begin{document}

\begin{abstract}
We study propagation of the Gabor wave front set for a Schr\"odinger equation with a Hamiltonian that is the Weyl quantization of a quadratic form with non-negative real part. 
We point out that the singular space associated to the quadratic form plays a crucial role for the understanding of this propagation. We show that the Gabor singularities of the solution to the equation for positive times are always contained in the singular space, and that they propagate in this set along the flow of the Hamilton vector field associated to the imaginary part of the quadratic form.
As an application we obtain for the heat equation a sufficient condition on the Gabor wave front set of the initial datum tempered distribution that implies regularization to Schwartz regularity for positive times. 
\end{abstract}

\keywords{Schr\"odinger equation, heat equation, propagation of singularities, Gabor wave front set, singular space.
MSC 2010 codes: 35A18, 35A21, 35Q40, 35Q79, 35S10}

\maketitle

\section{Introduction}

We study in this work evolution equations associated to quadratic operators. This class of operators stands for pseudodifferential operators
$$
q^w(x,D)u(x)  =(2\pi)^{-d} \int_{\rr{2d}} e^{i \la x-y, \xi \ra} q \left(\frac{x+y}{2},\xi \right) \, u(y) \, dy \, d\xi,
$$
defined by the Weyl quantization of complex-valued quadratic symbols $q(x,\xi)$ on the phase space $T^*\rr{d}$.
These operators are differential operators with simple and fully explicit expressions since the Weyl quantization of the quadratic symbol
$x^{\alpha} \xi^{\beta}$, with $(\alpha,\beta) \in \mathbb{N}^{2d}$, $|\alpha+\beta|=2$, is given by
$$(x^{\alpha} \xi^{\beta})^w=\textrm{Op}^w(x^{\alpha} \xi^{\beta})=\frac{x^{\alpha}D_x^{\beta}+D_x^{\beta} x^{\alpha}}{2},$$
with $D_x=i^{-1}\partial_x$. As the Weyl symbols are allowed to be complex-valued, it is a class of non-selfadjoint operators.

The maximal closed realization of a quadratic operator $q^w(x,D)$ on $L^2(\rr{d})$, that is, the operator equipped with the domain
$$\big\{u \in L^2(\rr{d}) : \ q^w(x,D)u \in L^2(\rr{d})\big\},$$
where $q^w(x,D)u$ is defined in the distribution sense, is known to coincide with the graph closure of its restriction to the Schwartz space~\cite{Hormander2} (pp.~425--426),
$$q^w(x,D) : \mathscr{S}(\rr{d}) \rightarrow \mathscr{S}(\rr{d}).$$
Classically, to any quadratic form defined on the phase space 
$$q : T^*\rr{d} \rightarrow \mathbb{C},$$
is associated a matrix $F \in \cc {2d \times 2d}$ called its Hamilton map, or its fundamental matrix, which is defined as the unique matrix satisfying the identity
$$\forall  (x,\xi) \in T^*\rr{d},\forall (y,\eta) \in T^*\rr{d}, \quad q((x,\xi),(y,\eta))=\sigma((x,\xi),F(y,\eta)),$$
with $q(\cdot,\cdot)$ the polarized form associated to the quadratic form $q$, where $\sigma$ stands for the standard symplectic form on $T^*\rr d$. This Hamiltonian matrix is given by $F=\J Q$ where 
\begin{equation*}
\J =
\left(
\begin{array}{cc}
0 & I_d \\
-I_d & 0
\end{array}
\right), 
\end{equation*}
and $Q \in \cc {2d \times 2d}$ is the symmetric matrix that defines $q$ by $q(x,\xi) = \la (x,\xi), Q(x,\xi) \ra$.

In~\cite{Hitrik1}, the notion of singular space was introduced by Hitrik and the first author by pointing out the existence of a particular vector subspace in the phase space $T^*\rr d$, which is intrinsically associated to a quadratic symbol $q$, and defined as the following finite intersection of kernels
\begin{equation}\label{h1}
S=\Big(\bigcap_{j=0}^{2d-1} \Ker \big[\re F(\im F)^j \big]\Big) \cap T^*\rr d \subseteq T^*\rr d, 
\end{equation}
where $\textrm{Re }F$ and $\textrm{Im }F$ stand for the real and imaginary parts of the Hamilton map $F$ associated to the quadratic symbol $q$,
$$\textrm{Re }F=\frac{1}{2}(F+\overline{F}), \quad \textrm{Im }F=\frac{1}{2i}(F-\overline{F}),$$
which are respectively the Hamilton maps of the quadratic forms $\re q$ and $\im q$.

As pointed out in \cite{Hitrik1,kps11,kps21,viola1}, the singular space is playing a basic role in understanding the spectral and hypoelliptic properties of non-elliptic quadratic operators, as well as the spectral and pseudospectral properties of certain classes of degenerate doubly characteristic pseudodifferential operators~\cite{kps3,kps4,viola}.

When the Weyl symbol of a quadratic operator $q^w(x,D)$ has a non-negative real part $\textrm{Re }q \geq 0$, this operator is maximal accretive and generates a contraction semigroup $(e^{-tq^w(x,D)})_{t \geq 0}$ on $L^2(\rr{d})$. In~\cite{Hitrik1}, the singular space was shown to play a key role in understanding the smoothing properties of semigroups generated by accretive quadratic operators. Indeed, it was shown in~\cite[Theorem~1.2.1]{Hitrik1} that the semigroup of a (possibly non-elliptic) accretive quadratic operator enjoys regularizing properties of Schwartz type for any positive time
$$\forall u_0 \in L^2(\rr{d}), \forall t>0, \quad e^{-tq^w(x,D)}u_0 \in \mathscr{S}(\rr{d}),$$
when its singular space is zero. This condition $S=\{0\}$ actually implies that the quadratic operator enjoys some subelliptic properties \cite[Theorem~1.2.1]{kps21}. It holds for instance for some kinetic operators as the Kramers--Fokker--Planck operator 
$$K=-\Delta_v+\frac{|v|^2}{4}+v\cdot\partial_x-\nabla V(x)\cdot\partial_v, \quad (x,v) \in \rr{2},$$
with a quadratic potential $V(x)=ax^2$, $a \in \mathbb{R} \setminus \{0\}$; some operators appearing in models of finite-dimensional Markovian approximation of the general Langevin equation, or in chains of oscillators coupled to heat baths~\cite[Section~4]{kps11}.

When the singular space $S$ is non-zero but still has a symplectic structure in the sense that the restriction of the canonical symplectic form to the singular space $\sigma|_{S}$ is non-degenerate, \cite[Theorem~1.2.1]{Hitrik1} shows more generally that the semigroup enjoys some smoothing properties of Schwartz type but only in the directions given by the orthogonal complement of the singular space in $T^*\rr d$ with respect to the symplectic form, 
$$
S^{\sigma \perp}=\{(x,\xi) \in T^*\rr d : \ \forall (y,\eta) \in S, \ \sigma((x,\xi),(y,\eta))=0\}. 
$$
This means that for all $t>0$, $N \in \no$ and $u \in L^2(\rr d)$,
$$(1+|x'|^2+|D_{x'}|^2)^Ne^{-t q^w(x,D)}u \in L^2(\rr d),$$ 
if $(x',\xi')$ are some linear symplectic coordinates on the symplectic space $S^{\sigma \perp}$.

In the present work, we study the general case when the singular space of a quadratic operator does not enjoy necessarily a symplectic structure. We aim at understanding what remains of the Schwartz type regularizing properties of the semigroup $(e^{-tq^w(x,D)})_{t \geq 0}$, and how the lack of Schwartz regularity of the initial datum may propagate in time.

We consider the initial value Cauchy problem for Schr\"odinger equations of the form
\begin{equation}\label{cp}
\left\{
\begin{array}{rl}
\partial_t u(t,x) + q^w(x,D) u (t,x) & = 0, \qquad t \geq 0, \quad x \in \rr d, \\
u(0,\cdot) & = u_0,  
\end{array}
\right.
\end{equation}
where the initial datum $u_0$ is more generally a tempered distribution on $\rr d$ and $q^w(x,D)$ is a quadratic operator whose Weyl symbol has a non-negative real part $\textrm{Re }q \geq 0$. The family of equations (\ref{cp}) includes for example the proper Schr\"odinger equation when $q(x,\xi)= i |\xi|^2$, or the heat equation if $q(x,\xi) = |\xi|^2$.

We aim at studying how the Gabor wave front set $WF(u_0)$ of the initial datum $u_0$ is propagated to the solution $u(t,\cdot)$ at positive time $t>0$. The Gabor wave front set (or singularities) was introduced by H\"ormander \cite{Hormander1} and measures the directions in the phase space in which a tempered distribution does not behave like a Schwartz function. It is hence empty if and only if a distribution that is a priori tempered is in fact a Schwartz function. The Gabor wave front set thus measures global regularity in the sense of both smoothness and decay at infinity. 

Our starting point is the case of a purely imaginary-valued quadratic form 
$q=i\textrm{Im }q.$ 
It corresponds to the case when the singular space is equal to the whole phase space $T^*\rr{d}$. With $e^{- t q^w(x,D)}$ the solution operator, it is then a known consequence of the metaplectic representation (see e.g. \cite{Cordero5}) that there is exact propagation of Gabor singularities as follows 
\begin{equation*}
WF(e^{- t q^w(x,D)} u_0) = e^{-2 i t F} WF(u_0), \qquad t \in \ro, \ u_0 \in \cS'(\rr d), 
\end{equation*}
where $F$ stands for the Hamilton map of $q$. We notice in this case that there is no regularizing effect of Schwartz type and that the Gabor singularities propagate along the curves given by the mapping $t \mapsto e^{-2 i t F}$, which corresponds to the flow of the Hamilton vector field 
$$H_{\textrm{Im}q}=\frac{\partial \im q}{\partial \xi} \cdot \frac{\partial}{\partial_x}-\frac{\partial \im q}{\partial x} \cdot \frac{\partial}{\partial_{\xi}},$$
associated to the imaginary part of the symbol
$$\forall t \in \mathbb{R},  \quad \exp(tH_{\textrm{Im}q})=e^{2t \im F}=e^{-2 i t F}.$$

The main result in this article is Theorem~\ref{th1}. When the Weyl symbol of a quadratic operator $q^w(x,D)$ has a non-negative real part $\textrm{Re }q \geq 0$, this result shows that all the Gabor singularities outside the singular space are smoothed out for any positive time $t>0$. More specifically, it proves that the following microlocal inclusions of the Gabor wave front sets  
\begin{align*}
WF\big(e^{-tq^w(x,D)}u_0\big) \subseteq \left( \exp(tH_{\textrm{Im}q})\big(WF(u_0) \cap S\big) \right) \cap S,
\end{align*}
hold for any positive time $t>0$. These microlocal inclusions point out that the possible Gabor singularities of the solution are localized in the singular space and as in the first case, do propagate along the curves given by the flow of the Hamilton vector field $H_{\textrm{Im}q}$ associated to the imaginary part of the symbol.

In Corollary~\ref{co1} we notice further that in the particular case when the Weyl symbol of the operator satisfies the null Poisson bracket condition $\{q, \overline q \}=0$ then the matrices $\re F$ and $\im F$ commute and the result of Theorem~\ref{th1} simply reads as the following microlocal inclusion
\begin{equation*}
WF(e^{- t q^w(x,D)}u_0) 
\subseteq  \left(  \exp(tH_{\textrm{Im}q})\big( WF(u_0) \cap \Ker ( \re F) \big) \right) \cap \Ker (\re F),
\end{equation*}
holds for any positive time $t>0$.

For the heat equation 
\begin{equation*}
\partial_t u(t,x) - \Delta_x u(t,x) = 0, \quad t \geq 0, \quad x \in \rr d, 
\end{equation*}
we have $\{q, \overline q \}=0$ and $\Ker (\re F) = \rr d \times \{0\}$, so the latter corollary gives
\begin{equation}\label{Gaborpropheat}
WF(e^{-t q^w(x,D)} u_0) 
\subseteq  WF(u_0) \cap ( \rr d \times \{ 0 \} ), \quad t > 0, \quad u \in \cS'(\rr d). 
\end{equation}
The heat equation is thus immediately globally regularizing
\begin{equation*}
e^{- t q^w(x,D)} u_0 \in \cS(\rr d), \quad t > 0, 
\end{equation*}
provided the Gabor wave front set of the initial datum $u_0 \in \cS'(\rr d)$ satisfies
\begin{equation*}
WF(u_0) \cap ( \rr d \times \{ 0 \} ) = \emptyset. 
\end{equation*}
Note that 
\begin{equation*}
\rr d \setminus \{ 0 \} \times \{ 0 \} = WF( 1 ). 
\end{equation*}
The heat propagator is thus immediately regularizing if $WF(u_0)$ is disjoint from $WF(1)$. 

Our results on the propagation of the Gabor wave front set are proved by a combination of the following ingredients. 
We rely heavily on H\"ormander's results \cite{Hormander1,Hormander2} concerning the Gabor wave front set, and the theory of Gaussian oscillatory integrals and their associated positive Lagrangians in $T^* \cc d$. 
This class of oscillatory integrals have the form
\begin{equation*}
u(x) = \int_{\rr N} e^{i p(x,\theta)} d \theta, \quad x \in \rr d,
\end{equation*}
where $p$ is a quadratic form with certain properties, among which $\im p \geq 0$. 
Each such oscillatory integral is associated with a positive Lagrangian in $T^* \cc d$, uniquely determined by $u$ up to multiplication with $\co \setminus 0$. 

One basic building block for our results is the fact that the solution operator $e^{-t q^w(x,D)}$ of the Schr\"odinger equation has a Schwartz kernel that is an oscillatory integral associated to a positive Lagrangian $\lambda \subseteq T^* \cc {2d}$ that is  defined in terms of $F$. This theory is developed in \cite{Hormander2}. 

Another essential component is the microlocal inclusion from \cite{Hormander1} for the Gabor wave front set, that reads roughly
\begin{equation}\label{WFpropbasic}
WF (\cK u) \subseteq WF(K)' \circ WF(u)
\end{equation}
for a linear operator $\cK: \cS (\rr d) \mapsto \cS' (\rr d)$ with Schwartz kernel $K \in \cS' (\rr {2d})$, expressed with the twisted Gabor wave front set $WF(K)'$ of the Schwartz kernel, where the fourth coordinate is reflected. 

The basic idea of our proofs is to combine these two building blocks with the following result that we prove (see Theorem \ref{WFinclusion}): The Gabor wave front set of an oscillatory integral on $\rr d$ is contained in the corresponding positive Lagrangian intersected with $T^* \rr d$. This gives information on the Gabor wave front set of the Schwartz kernel of the propagator $\cK=e^{- t q^w(x,D)}$, which when inserted into \eqref{WFpropbasic} gives results on propagation of Gabor singularities. 

We note that the recent papers and preprints \cite{Cordero5,Cordero6,Nicola2} treat propagation of Gabor-type wave front sets for various classes of Schr\"odinger equations. These classes cater for the presence of potentials with certain properties and symbols $q$ of greater generality than quadratic forms, but they do not admit the quadratic form $q$ to have a nonzero real part. 
We also note the recent preprint \cite{Aleman1} containing interesting ideas on Schr\"odinger equations of more general type than we study. These ideas are different from those explored in this paper.  

The papers \cite{Weinstein1} by Weinstein and \cite{Zelditch1} by Zelditch are interesting references for propagation of singularities for Schr\"odinger equations. In these papers the propagation of the classical $C^\infty$ wave front set (cf. \cite[Chapter 8]{Hormander0}) is studied, and in \cite{Weinstein1} also a notion called metawavefront set, that is not the same as the Gabor wave front set although they have some features in common. For the classical wave front set, it is shown that the propagator for the harmonic oscillator with a perturbation under certain circumstances smooths out singularities except at periodically recurring time points when the singularities reappear. 

These works deal with Schr\"odinger equations with real-valued Hamiltonians. We consider here Schr\"odinger equations with complex-valued Hamiltonians. Our main result actually displays how non-trivial interplays due to the possible non-commutativity between the real and imaginary parts of the Hamiltonian matrix, accounting for the structure of the singular space, imply Schwartz smoothing properties of Schr\"odinger equations with complex-valued Hamiltonians. 

The paper is organized in the following manner. Section \ref{prelim} fixes notation and introduces background results. 
In Section \ref{formulation}, the Schr\"odinger equation is formulated and we briefly discuss the semigroup theory underlying the construction of the solution operator. 
Section \ref{oscintsec} treats H\"ormander's theory of oscillatory integrals with quadratic phase functions, which is a subspace of distributions designated Gaussians. We give a fairly detailed (and somewhat redundant to \cite{Hormander2}) treatment of this theory.
In particular we prove a parametrization of positive Lagrangians in $T^* \cc d$ which is needed for  
the crucial Theorem \ref{WFinclusion}. 

In Section \ref{kernelsec}, we connect Theorem \ref{WFinclusion} with H\"ormander's result \cite{Hormander2} on the Schwartz kernel of the Schr\"odinger propagator to give an inclusion for the Gabor wave front set of the Schwartz kernel of the propagator (Theorem \ref{WFkernel}). 

In Section \ref{generalcasesec}, we prove Theorem \ref{th1} concerning the propagation of Gabor singularities
in terms of the singular space and $\im F$.
We also state Corollary \ref{co1} which is a consequence of Theorem \ref{th1} and the additional assumption $\{q, \overline q \}=0$. 

In Section \ref{secparticulareq}, we give a variety of examples of equations and how they propagate the Gabor singularities. 

Finally in Section \ref{secexamples}, we give some examples of tempered distributions on $\ro$ and their Gabor wave front sets. One of these examples, which is in fact the Airy function, shows that the heat equation microlocal inclusion \eqref{Gaborpropheat} may be strict.

\section{Preliminaries}\label{prelim}

The gradient of a function $f$ with respect to the variable $x \in \rr d$ is denoted by $f'_x$ and
the Hessian with respect to $x,y$ is denoted $f_{x y}''$. 
The gradient operator is denoted $\partial/\partial x$. 
The Fourier transform of $f \in \cS(\rr d)$ (the Schwartz space) is normalized as
$$
\mathscr{F} f(\xi) = \wh f(\xi) = \int_{\rr d} f(x) e^{- i \la x, \xi \ra} dx,
$$
where $\la x, \xi \ra = x \cdot \xi = \sum_{j=1}^d x_j \xi_j$ denotes the inner product on $\rr d$. 
The same notation is used for the inner product on $\cc d$ which is linear in both variables (note that it is not sesquilinear). 

The Japanese bracket is $\eabs{x} = (1+|x|^2)^{1/2}$ for $x \in \rr d$. 
A closed ball in $\rr d$ of radius $\ep>0$ is denoted $B_\ep (\rr d)= \{ x \in \rr d: \, |x| \le \ep \}$.
For a matrix $A$ with real or complex entries $A \geq 0$ means positive semidefinite and $A > 0$ means positive definite. 
The transpose of $A$ is $A^t$. 
In estimates we denote by $f (x) \lesssim g (x)$ that $f(x) \leq C g(x)$ holds for some constant $C>0$ and all $x$ in the domain of $f$ and $g$. 

We denote the translation operator by $T_x f(y)=f(y-x)$, the modulation operator by $M_\xi f(y)=e^{i \la y, \xi \ra} f(y)$, $x,y,\xi \in \rr d$, and the
phase space translation operator by $\Pi(z) = M_\xi T_x$, $z=(x,\xi) \in \rr {2d}$. 
Given a window function $\varphi \in \cS(\rr d) \setminus \{ 0 \}$, the short-time Fourier transform (STFT) (cf. \cite{Grochenig1}) of $f \in \cS'(\rr d)$ is defined by
\begin{equation}\nonumber
V_\varphi f(z) = ( f, \Pi(z) \varphi ), \quad z \in \rr {2d},
\end{equation}
where $(\cdot,\cdot)$ denotes the conjugate linear action of $\cS'$ on $\cS$,
consistent with the inner product $(\cdot,\cdot)_{L^2}$ which is conjugate linear in the second argument.
The function $z \mapsto V_\varphi f(z)$ is smooth and its modulus is bounded by $C \eabs{z}^k$ for all $z \in \rr {2d}$ for some $C,k \geq 0$.
If $\varphi \in \cS(\rr d)$, $\| \varphi \|_{L^2}=1$ and $f \in \cS'(\rr d)$,
the STFT inversion formula reads (cf. \cite[Corollary 11.2.7]{Grochenig1})
\begin{equation}\label{STFTrecon}
(f,g ) = (2 \pi)^{-d} \int_{\rr {2d}} V_\varphi f(z) ( \Pi(z) \varphi,g ) \, dz, \quad g \in \cS(\rr d).
\end{equation}

The Weyl quantization (cf. \cite{Folland1,Hormander0,Shubin1}) is the map from a symbol $a$ to an operator defined by
\begin{equation}\nonumber
a^w(x,D) f(x) = (2 \pi)^{-d} \int_{\rr {2d}} e^{i \la x-y,\xi \ra} a \left( \frac{x+y}{2},\xi \right)  \, f(y) \, dy \, d \xi
\end{equation}
for $a \in \cS(\rr {2d})$ and $f \in \cS(\rr d)$. The latter conditions can be relaxed in various ways.
By the Schwartz kernel theorem, any continuous linear operator $\cS(\rr d) \mapsto \cS'(\rr d)$
can be written as the Weyl quantization for a unique symbol $a \in \cS'(\rr {2d})$.

\begin{defn}\label{shubinclasses1}\cite{Shubin1}
For $m\in \ro$, the Shubin symbol class $G^m$ is the subspace of all
$a \in C^\infty(\rr {2d})$ such that for every
$\alpha,\beta \in \nn d$ 
\begin{equation}\label{symbolestimate1}
|\partial_x^\alpha \partial_\xi^\beta a(x,\xi)| 
\lesssim \langle (x,\xi) \rangle^{m-|\alpha|-|\beta|}, \quad (x,\xi)\in \rr {2d}, 
\end{equation}
and $G^m$ is a Fr\'echet space with respect to the seminorms 
\begin{equation*}
\sup_{(x,\xi) \in \rr {2d}} \la (x,\xi) \ra^{-m+|\alpha|+|\beta|} \left| \partial_x^\alpha \partial_\xi^\beta a(x,\xi) \right|, \quad (\alpha,\beta) \in \nn {2d}.
\end{equation*}
\end{defn}

\begin{defn}\label{hormanderclasses}\cite{Hormander0}
For $m\in \ro$, $0 \leq \rho \leq 1$, $0 \leq \delta < 1$, the H\"ormander symbol class $S_{\rho,\delta}^m$ is the subspace of all
$a \in C^\infty(\rr {2d})$ such that for every
$\alpha,\beta \in \nn d$ 
\begin{equation}\label{symbolestimate2}
|\partial_x^\alpha \partial_\xi^\beta a(x,\xi)| 
\lesssim \eabs{\xi}^{m - \rho|\beta| + \delta |\alpha|}, \quad (x,\xi)\in \rr {2d}, 
\end{equation}
and $S_{\rho,\delta}^m$ is a Fr\'echet space with respect to the seminorms 
\begin{equation*}
\sup_{(x,\xi) \in \rr {2d}} \eabs{\xi}^{-m + \rho|\beta| - \delta |\alpha|} \left| \partial_x^\alpha \partial_\xi^\beta a(x,\xi) \right|, \quad (\alpha,\beta) \in \nn {2d}.
\end{equation*}
\end{defn}

The following definition involves conic sets in the phase space $T^* \rr d \simeq \rr {2d}$. 
A set is conic if it is invariant under multiplication with positive reals. 
Note the difference to the frequency-conic sets that are used in the definition of the (classical) $C^\infty$ wave front set \cite{Hormander0}. 

\begin{defn}\label{noncharacteristic2}
Given $a \in G^m$, a point in the phase space $z_0 \in T^* \rr d \setminus \{ (0,0) \}$ is called non-characteristic for $a$ provided there exist $A,\ep>0$ and an open conic set $\Gamma \subseteq T^* \rr d \setminus \{ (0,0) \}$ such that $z_0 \in \Gamma$ and
\begin{equation}\label{noncharlowerbound2}
|a(z )| \geq \ep \eabs{z}^m, \quad z \in \Gamma, \quad |z| \geq A.
\end{equation}
\end{defn}

The Gabor wave front set is defined as follows. 

\begin{defn}\label{wavefront1}
\cite{Hormander1}
If $u \in \cS'(\rr d)$ then the Gabor wave front set $WF(u)$ is the set of all $z \in T^*\rr d \setminus \{ (0,0) \}$ such that $a \in G^m$ for some $m \in \ro$ and $a^w(x,D) u \in \cS$ implies that $z \in \charac(a)$.
\end{defn}

According to \cite[Proposition 6.8]{Hormander1} and \cite[Corollary 4.3]{Rodino1}, the Gabor wave front set can be characterized microlocally by means of the STFT as follows. 

If $u \in \cS'(\rr d)$ and $\varphi \in \cS(\rr d) \setminus \{0 \}$ then $z_0 \in T^*\rr d \setminus \{ (0,0) \}$ satisfies $z_0 \notin WF(u)$ if and only if there exists an open conic set $\Gamma_{z_0} \subseteq T^*\rr d \setminus \{ (0,0) \}$ containing $z_0$ such that
\begin{equation*}
\sup_{z \in \Gamma_{z_0}} \eabs{z}^N |V_\varphi u(z)| < +\infty, \quad \forall N \geq 0.
\end{equation*}

The Gabor wave front set has properties the most important of which includes the following list. 

\begin{enumerate}

\item If $u \in \cS'(\rr d)$ then $WF(u) = \emptyset$ if and only if $u \in \cS (\rr d)$ \cite[Proposition 2.4]{Hormander1}.

\item If $u \in \cS'(\rr d)$ and $a \in G^m$ then
\begin{align*}
WF( a^w(x,D) u) 
& \subseteq WF(u) \cap \conesupp (a) \\
& \subseteq WF( a^w(x,D) u) \ \cup \ \charac (a). 
\end{align*}
Here the conic support $\conesupp (a)$ of $a \in \cD'(\rr n)$  is the set of all
$x \in \rr n \setminus \{ 0 \}$ such that any conic open set $\Gamma_x \subseteq \rr n \setminus \{ 0 \}$ containing $x$ satisfies:
\begin{equation*}
\overline{\supp (a) \cap \Gamma_x} \quad \mbox{is not compact in} \quad \rr n.
\end{equation*}

\item If $a \in S_{0,0}^0$ and $u \in \cS'(\rr d)$ then by \cite[Theorem 5.1]{Rodino1}
\begin{equation}\label{microlocal2}
WF(a^w(x,D) u) \subseteq WF(u).
\end{equation}
In particular $WF(\Pi(z) u) = WF(u)$ for any $z \in \rr {2d}$.  

\end{enumerate}

As three basic examples of the Gabor wave front set we have (cf. \cite[Example 6.4--6.6]{Rodino1})
\begin{equation}\label{example1}
WF(\delta_x) = \{ 0 \} \times (\rr d \setminus \{ 0 \}), \quad x \in \rr d,  
\end{equation}
\begin{equation}\label{example2}
WF(e^{i \la \cdot,\xi \ra}) = (\rr d \setminus \{0\}) \times  \{ 0 \}, \quad \xi \in \rr d, 
\end{equation}
and 
\begin{equation}\label{example3}
WF(e^{i \la x, A x \ra/2 } ) = \{ (x, Ax): \, x \in \rr d \setminus 0 \}, \quad A \in \rr {d \times d} \quad \mbox{symmetric}. 
\end{equation}

A symplectic vector space over $\ro$ or $\co$ is an even-dimensional vector space equipped with a nondegenerate antisymmetric bilinear form. 
Two basic examples are $T^* \rr d$ and $T^* \cc d$, both equipped with the canonical symplectic form 
\begin{equation}\label{cansympform}
\sigma((x,\xi), (x',\xi')) = \la x' , \xi \ra - \la x, \xi' \ra
\end{equation}
where $(x,\xi), (x',\xi') \in T^* \rr d$ or $(x,\xi), (x',\xi') \in T^* \cc d$ respectively. 
Using the matrix 
$$
\J =
\left(
\begin{array}{cc}
0 & I_d \\
-I_d & 0
\end{array}
\right) \in \rr {2d \times 2d}
$$
the symplectic form can be expressed with the inner product on $\rr {2d}$ or $\cc {2d}$ as
$$
\sigma((x,\xi), (x',\xi')) = \la \J (x,\xi), (x',\xi') \ra. 
$$

A Lagrangian (subspace) $\lambda \subseteq V$ of a symplectic vector space $V$ of dimension $2d$ with symplectic form $\sigma$ satisfies 
\begin{equation*}
\lambda = \{ X \in V: \, \sigma(X,Y) = 0 \ \forall Y \in \lambda \}  
\end{equation*}
and has dimension $d$. 
Thus $\sigma(X,Y) = 0$ when $X,Y \in \lambda$. 

The real symplectic group $\Sp(d,\ro)$ consists of all matrices $\chi \in \GL( 2 d, \ro)$ that satisfy
\begin{equation}\label{symplecticgroup}
\sigma(\chi z,\chi z') = \sigma(z,z')
\end{equation}
for all $z,z' \in T^* \rr {d}$. 
The complex symplectic group $\Sp(d,\co)$ consists of all matrices $\chi \in \GL( 2 d, \co)$ that satisfy \eqref{symplecticgroup} for all $z,z' \in T^* \cc {d}$. 

To each symplectic matrix $\chi \in \Sp(d,\ro)$ is associated a unitary operator $\mu(\chi)$ on $L^2(\rr d)$, determined up to a complex factor of modulus one, such that
\begin{equation}\label{symplecticoperator}
\mu(\chi)^{-1} a^w(x,D) \, \mu(\chi) = (a \circ \chi)^w(x,D), \quad a \in \cS'(\rr {2d})
\end{equation}
(cf. \cite{Folland1,Hormander0}).
The operator $\mu(\chi)$ is an homeomorphism on $\cS$ and on $\cS'$.

The mapping $\Sp(d,\ro) \ni \chi \mapsto \mu(\chi)$ is called the metaplectic representation \cite{Folland1,Taylor1}.
It is in fact a representation of the so called $2$-fold covering group of $\Sp(d,\ro)$, which is called the metaplectic group and denoted $\Mp(d,\ro)$.
The metaplectic representation satisfies the homomorphism relation only modulo a change of sign:
\begin{equation*}
\mu( \chi \chi') = \pm \mu(\chi ) \mu(\chi' ), \quad \chi, \chi' \in \Sp(d,\ro).
\end{equation*}

According to \cite[Proposition~2.2]{Hormander1} the Gabor wave front set is symplectically invariant as follows
for $u \in \cS'(\rr d)$,
\begin{equation*}
(x,\xi) \in WF(u) \quad \Longleftrightarrow \quad \chi(x,\xi) \in WF( \mu(\chi) u) \quad \forall \chi \in \Sp(d, \ro),
\end{equation*}
or, in short,
\begin{equation}\label{symplecticinvarianceWF}
WF( \mu(\chi) u) = \chi WF(u), \quad \chi \in \Sp(d, \ro), \quad u \in \cS'(\rr d).
\end{equation}

\section{Problem formulation and the solution operator}
\label{formulation}

We study the initial value problem for a Schr\"odinger equation of the form
\begin{equation}\label{schrodeq}
\left\{
\begin{array}{rl}
\partial_t u(t,x) + q^w(x,D) u (t,x) & = 0, \\
u(0,\cdot) & = u_0, 
\end{array}
\right.
\end{equation}
with $u_0 \in \cS'(\rr d)$, $t \geq 0$, $x \in \rr d$, where $q^w(x,D)$ is a quadratic operator whose Weyl symbol 
$$q(x,\xi)=\langle (x,\xi),Q(x,\xi) \rangle, \quad (x,\xi) \in T^*\rr{d},$$
has a non-negative real part $\textrm{Re }Q \geq 0$.
The equation \eqref{schrodeq} is solved for $u_0 \in L^2(\rr d)$ and $t \geq 0$ by 
$$
u(t,x) = e^{-t q^w(x,D)} u_0(x),
$$
where the solution operator, or propagator, $e^{- t q^w(x,D)}$ is the contraction semigroup that is generated by the operator $-q^w(x,D)$. 
A contraction semigroup is a strongly continuous semigroup with $L^2$ operator norm  $\leq 1$ for all $t \geq 0$ \cite{Yosida1}. 
The fact that $-q^w(x,D)$, or more precisely its closure $M_{-q}$ as an unbounded linear operator in $L^2(\rr d)$, generates such a semigroup is proved in \cite[pp. 425--26]{Hormander2}. The conclusion follows from the fact that $M_{- q}$ and its adjoint $M_{-\overline q}$ are dissipative operators \cite{Yosida1}. This means for $M_{-q}$ that
$$
\re (M_{-q} u,u) = (M_{-\re q} u,u) \leq 0, \quad u \in D(M_{-q}), 
$$
which follows from the assumption $\re Q \geq 0$. 
Here $D(M_{-q})$ denotes the domain of $M_{-q}$. On the other hand, it was shown in~\cite{Hormander2} (Proposition~5.8 and Theorem~5.12) that the solution operator $e^{-tq^w(x,D)}$ is a Fourier integral operator defining a continuous mapping on the Schwartz space
$$e^{-tq^w(x,D)}: \mathscr{S}(\rr{d}) \rightarrow \mathscr{S}(\rr{d}),$$ 
which can be extended by duality as a continuous mapping on the space of tempered distributions
$$e^{-tq^w(x,D)}: \mathscr{S}'(\rr{d}) \rightarrow \mathscr{S}'(\rr{d}).$$ 
We are interested in the propagation of the Gabor wave front set for the Schr\"odinger propagator $e^{-t q^w(x,D)}$, that is, we want to find microlocal inclusions for the Gabor wave front set
$$
WF(e^{-t q^w(x,D)} u_0)
$$
in terms of $WF(u_0)$, $F$ and $t \geq 0$, for $u_0 \in \cS'(\rr d)$. 

If $\re Q=0$ then exact propagation is given by means of the metaplectic representation. In fact, if 
$\re Q=0$ then $e^{-t q^w(x,D)}$ is a group of unitary operators, and we have by \cite[Theorem 4.45]{Folland1}
$$
e^{- t q^w(x,D)} = \mu(e^{-2 i t F}), \quad t \in \ro. 
$$
Note that $F$ is purely imaginary in this case and $i F \in \ssp(d,\ro)$, the symplectic Lie algebra, which implies that $e^{-2 i t F} \in \Sp(d,\ro)$ for any $t \in \ro$ \cite{Folland1}.
According to \eqref{symplecticinvarianceWF} we thus have 
\begin{equation}\label{realcase}
WF(e^{- t q^w(x,D)} u_0) = e^{- 2 i t F} WF(u_0), \quad t \in \ro, \ u_0 \in \cS'(\rr d). 
\end{equation}

\section{Oscillatory integrals, Gaussians, associated Lagrangians and the Gabor wave front set}
\label{oscintsec}

In this section we give an account of a class of oscillatory integrals with quadratic phase functions and trivial amplitudes introduced by H\"ormander \cite{Hormander2}. 
To give a coherent picture we rededuce some of the results in \cite{Hormander2}. 
The main novelty is Theorem \ref{WFinclusion} which is an inclusion for the Gabor wave front set of an oscillatory integral. 
This class of oscillatory integrals will be useful in Section \ref{kernelsec}, since H\"ormander has proved that the Schwartz kernel of the propagator for the Schr\"odinger equation belongs to this class. 

The oscillatory integrals have the form 
\begin{equation}\label{oscillint1}
u(x) = \int_{\rr N} e^{i p(x,\theta)} \, d \theta, \quad x \in \rr d, 
\end{equation}
where $p$ is a complex-valued quadratic form on $\rr {d + N}$. 
Thus 
\begin{equation}\label{pform}
p(x,\theta) = \la (x, \theta), P (x, \theta) \ra, \quad x \in \rr d, \quad \theta \in \rr N, 
\end{equation}
where $P \in \cc {(d+N) \times (d+N)}$ is symmetric. 
We write the matrix $P$ as
\begin{equation}\label{Pmatrix}
P=\left(
\begin{array}{ll}
P_{xx} & P_{x \theta} \\
P_{\theta x} & P_{\theta \theta} 
\end{array}
\right)
\end{equation}
where $P_{xx} \in \cc {d \times d}$, $P_{x \theta} \in \cc {d \times N}$, $P_{\theta \theta} \in \cc {N \times N}$, and $P_{\theta x} = P_{x \theta}^t$. 

Provided the following two conditions hold for the matrix $P$, we call it a Gaussian generator matrix, for reasons that will become clear shortly. 

\begin{enumerate}

\item  $\im P \geq 0$; 

\item the row vectors of  the submatrix
\begin{equation*}
\left(
\begin{array}{ll}
P_{\theta x} & P_{\theta \theta} 
\end{array}
\right) \in \cc {N \times (d+N)}
\end{equation*}
are linearly independent over $\co$. 

\end{enumerate}

A statement equivalent to (2) is that the linear forms $\{ \partial p/\partial \theta_j \}_{j=1}^N$ be linearly independent over $\co$. 

It is clear that if assumption (1) is strengthened to $\im P > 0$ then the oscillatory integral \eqref{oscillint1} is absolutely convergent and $u \in L^1(\rr d)$. 
We will next show that when $\im P \geq 0$ the integral \eqref{oscillint1} still makes sense if it is interpreted as a  regularization, and then $u \in \cS'(\rr d)$. 

\begin{lem}\label{matrixlemma}
Suppose $p$ is a quadratic form on $\rr {d + N}$ defined by a symmetric Gaussian generator matrix $P \in \cc {(d+N) \times (d+N)}$. 
Then there exist matrices $A \in \cc {N \times d}$, $B \in \cc {N \times d}$ and $C \in \cc {N \times N}$ such that 
\begin{equation}\label{thetaformula}
\theta = A x + B p_x'(x,\theta) + C p_\theta '(x,\theta), \quad x\in \rr d, \quad \theta \in \rr N. 
\end{equation}
\end{lem}

\begin{proof}
The gradients $p_x'(x,\theta)$ and $p_\theta'(x,\theta)$ are the linear forms
\begin{equation*}
\begin{aligned}
& p_x'(x,\theta) = 2 
\left(
\begin{array}{ll}
P_{x x} & P_{x \theta} 
\end{array}
\right)
\left(
\begin{array}{l}
x \\
\theta 
\end{array}
\right), \\
& p_\theta'(x,\theta) = 2 
\left(
\begin{array}{ll}
P_{\theta x} & P_{\theta \theta} 
\end{array}
\right)
\left(
\begin{array}{l}
x \\
\theta 
\end{array}
\right), 
\end{aligned}
\end{equation*}
respectively. 
Thus we can rewrite \eqref{thetaformula} as 
\begin{equation}\label{reformulation1}
\left(A+2 B P_{x x} + 2 C P_{\theta x} \right) x + 
\left(2  
\left(
\begin{array}{ll}
B & C
\end{array}
\right)
\left(
\begin{array}{l}
P_{x \theta} \\
P_{\theta \theta} 
\end{array}
\right)
- I_N \right) \theta = 0
\end{equation}
for all $(x,\theta) \in \rr {d+N}$. 
The assumptions imply that the matrix 
\begin{equation*}
\left(
\begin{array}{l}
P_{x \theta} \\
P_{\theta \theta} 
\end{array}
\right)
\end{equation*}
has full rank equal to $N$. Therefore the matrices $B$ and $C$ can be chosen such that 
\begin{equation}\label{leftinverse}
2  \left(
\begin{array}{ll}
B & C
\end{array}
\right)
\left(
\begin{array}{l}
P_{x \theta} \\
P_{\theta \theta} 
\end{array}
\right)
= I_N. 
\end{equation}
If we then let $A=-2 (B P_{x x} + C P_{\theta x})$ it follows that \eqref{reformulation1} holds for all
$(x,\theta) \in \rr {d+N}$. 
\end{proof}

The following result is proved in \cite[Proposition 5.5]{Hormander2}. 

\begin{prop}\label{oscintprop}
Let $p$ be a quadratic form on $\rr {d + N}$ defined by a symmetric Gaussian generator matrix $P \in \cc {(d+N) \times (d+N)}$. 
Then the oscillatory integral \eqref{oscillint1} can be interpreted as a tempered distribution $u \in \cS'(\rr d)$ that agrees with the Lebesgue integral \eqref{oscillint1} when $\im P > 0$. 
\end{prop}

\begin{proof}
First suppose $\im P > 0$. Then the integral  \eqref{oscillint1} is absolutely convergent. 
By Lemma \ref{matrixlemma} we can write for any $1 \leq j \leq N$  
\begin{equation*}
\theta_j = \sum_{k=1}^{d} a_{j k} x_k + \sum_{k=1}^{d} b_{j k} \frac{\partial p}{\partial x_k} (x,\theta) + \sum_{k=1}^{N} c_{j k} \frac{\partial p}{\partial \theta_k} (x,\theta)
\end{equation*}
for some complex coefficients $(a_{j k})$, $(b_{j k})$, $(c_{j k})$.
Let $\fy \in \cS(\rr d)$. 
Multiplying with $(1-i \theta_j) (1-i \theta_j)^{-1}$ and integrating by parts give
\begin{align*}
& \iint_{\rr {d+N}} e^{i p(x,\theta)} \, \fy (x)  \, dx \, d \theta \\
& = \iint_{\rr {d+N}} e^{i p(x,\theta)} \left( 
1-i \sum_{k=1}^{d} a_{j k} x_k + \sum_{k=1}^{d} b_{j k} \frac{\partial}{\partial x_k} + \sum_{k=1}^{N} c_{j k} \frac{\partial}{\partial \theta_k} 
\right) \, 
\frac{\fy(x)}{1-i \theta_j}
 \, dx \, d \theta \\
& = \iint_{\rr {d+N}} e^{i p(x,\theta)} \left( 
\frac{L_j \, \fy (x)}{1-i \theta_j} + \frac{i c_{j j} \fy (x)}{(1-i \theta_j)^2} \right) \, dx \, d \theta, 
\end{align*}
where 
$L_j$ is the partial differential operator 
\begin{equation*}
L_j = 1-i \sum_{k=1}^{d} a_{j k} x_k + \sum_{k=1}^{d} b_{j k} \frac{\partial}{\partial x_k}. 
\end{equation*}
Repeating the same procedure once again, and then iterating over $j=1,\dots,N$, yield by induction
\begin{multline}\label{regularization1}
\iint_{\rr {d+N}} e^{i p(x,\theta)} \, \fy (x)  \, dx \, d \theta \\
 = \iint_{\rr {d+N}} e^{i p(x,\theta)} 
\sum_{|\alpha| + |\beta|=4N, \, \beta_1,\dots, \beta_N \geq 2}
C_{\alpha,\beta} \, \frac{L_N^{\alpha_N} \cdots L_1^{\alpha_1} \fy (x)}{(1-i\theta_1)^{\beta_1} \cdots {(1-i\theta_N)^{\beta_N}}}
\, dx \, d \theta,
\end{multline}
where $\alpha=(\alpha_1,\dots,\alpha_N) \in \nn N$, $\beta=(\beta_1,\dots,\beta_N) \in \nn N$ and $C_{\alpha,\beta} \in \co$ are constants. 

If we now relax the assumption $\im P > 0$ to $\im P \geq 0$ we see that the right hand side integral \eqref{regularization1} still converges absolutely, thanks to the terms $(1-i\theta_j)^{-\beta_j}$ where $\beta_j \geq 2$.
It follows that
\begin{equation*}
\la u, \fy \ra = \iint_{\rr {d+N}} e^{i p(x,\theta)} \, \fy (x)  \, dx \, d \theta 
\end{equation*}
defined by the right hand side of \eqref{regularization1} is a well defined element $u \in \cS'(\rr d)$ when $\im P \geq 0$. 
\end{proof}

It is shown in \cite{Hormander2} that a tempered distribution $u \in \cS'(\rr d)$ defined by the oscillatory integral \eqref{oscillint1} is a so called Gaussian, which justifies our designation of the corresponding matrix $P$ as a Gaussian generator. 

To define a Gaussian one needs the following concept. 
For $0 \neq u \in \cD'(\rr d)$ set
\begin{equation*}
\cL_u = \{ L(x,\xi)= \la a,\xi \ra + \la b,x \ra : \, a, b \in \cc d, \, L^w(x,D) u = 0 \}.
\end{equation*}
Thus $\cL_u$ is the space of linear symbols $L$ whose Weyl operator $L^w(x,D)$ acting on $u$ is zero. 
The distribution $u \in \cD'(\rr d)$ is called Gaussian if every $v \in \cD'(\rr d)$ such that $L^w(x,D) v = 0$ for all $L \in \cL_u$ is a complex multiple of $u$. 

As shown in \cite[Proposition 5.1]{Hormander2}, a Gaussian $u$ is, up to multiplication with a nonzero complex number, bijectively associated with a Lagrangian subspace $\lambda_u \subseteq T^* \cc d$, defined as
\begin{equation*}
\lambda_u = \bigcap_{L \in \cL_u } \{ (x,\xi) \in T^* \cc d: \, L(x,\xi) =0 \}. 
\end{equation*}
As the term indicates, a Gaussian is a distribution of the form
\begin{equation*}
C (\delta_0 \otimes e^q) \circ T
\end{equation*}
where $T \in \rr {d \times d}$ is an invertible matrix, $\delta_0$ is the delta function on $\rr k$, $0 \leq k \leq d$, $q$ is a quadratic form on $\rr {d-k}$, and $C \in \co \setminus 0$. 

A Gaussian is a tempered distribution exactly when the quadratic form $q$ has a non-positive real part. 
This happens if and only if the corresponding Lagrangian $\lambda \subseteq T^* \cc d$ is positive \cite{Hormander0b, Hormander2, Melin1}, which means
\begin{equation*}
i \sigma(\overline X,X) \geq 0, \quad X \in \lambda. 
\end{equation*}

It is proven in \cite[Proposition 5.6]{Hormander2} that an oscillatory integral of the form \eqref{oscillint1} does not have a uniquely determined integrand domain dimension $N$. It may happen that $N$ can be decreased, and if the quadratic form $p$ is modified accordingly the resulting modified oscillatory integral gives the same distribution $u \in \cS' (\rr d)$ times a nonzero complex constant.

According to \cite[Proposition 5.5]{Hormander2}, the Lagrangian corresponding to the oscillatory integral \eqref{oscillint1} is the positive Lagrangian 
\begin{equation}\label{lagrangian1}
\lambda = \{ (x, p_x'(x,\theta)) \in T^* \cc d: \, p_\theta'(x,\theta) = 0, \, (x,\theta) \in \cc {d+N} \} \subseteq T^* \cc d. 
\end{equation}

If a Lagrangian $\lambda \subseteq T^* \cc d$ is strictly positive, that is, 
\begin{equation*}
i \sigma(\overline X,X) > 0, \quad 0 \neq X \in \lambda,  
\end{equation*}
then by \cite[Proposition 21.5.9]{Hormander0} there exists a symmetric matrix $P \in \co^{d \times d} $ with $\im P > 0$ such that 
\begin{equation}\label{strictposlag1}
\lambda = \{ (x, 2 Px), \, x \in \cc d \}. 
\end{equation}
This matrix $P$ can be considered a particular case of matrices of the form \eqref{Pmatrix} with $N=0$. It satisfies the criteria for a Gaussian generator matrix (with condition (2) interpreted as non-existent when $N=0$), and  
\eqref{strictposlag1} is a particular case of a Lagrangian of the form \eqref{lagrangian1} with $p(x,\theta)=\la x, P x \ra$. 
The corresponding Gaussian 
$$
u(x) = e^{i \la x, P x \ra}, \quad x \in \rr d, 
$$ 
can be considered a degenerate case of an oscillatory integral of the type \eqref{oscillint1}. 

As we will see next, also a Lagrangian that is merely positive and not necessarily strictly positive can be given the form \eqref{lagrangian1}. 
In fact, a more particular parametrization of the Lagrangian exists, the properties of which are needed for the forthcoming inclusion of the Gabor wave front set of an oscillatory integral into the intersection of the corresponding positive Lagrangian and $T^* \rr d$ (Theorem \ref{WFinclusion}). 

First we need a lemma expressed in terms of the projection $\lambda_1$ of a Lagrangian $\lambda \subseteq T^* \cc d$ on the first $\cc d$ coordinate:
\begin{equation}\label{proj1coord}
\lambda_1 = \{ x \in \cc d: \, \exists \xi \in \cc d: \, (x,\xi) \in \lambda \} \subseteq \cc d,  
\end{equation}
which is a complex linear space of dimension $0 \leq k \leq d$. 
We split $\cc d$ variables as $x=(x',x'')$, $x' \in \cc k$, $x'' \in \cc N$ where $N=d-k$. 

\begin{lem}\label{lagrangianpure}
Let $\lambda \subseteq T^* \cc d$ be a positive Lagrangian and suppose 
\begin{equation}\label{subspace1}
\lambda_1 = \cc k \times \{ 0 \}. 
\end{equation}
If $k=0$ then $\lambda=\{0\} \times \cc d$, and if $1 \leq k \leq d$ then 
there exists  a matrix $Q \in \cc {k \times k}$ such that $\im Q \geq 0$ and 
\begin{equation}\label{lagrangerep}
\lambda = \{ (x',0,Qx',\xi''): \, x' \in \cc k, \, \xi'' \in \cc N \}. 
\end{equation}
\end{lem}

\begin{proof}
If $\lambda \subseteq T^* \cc d$ is a strictly positive Lagrangian then by the preceding discussion \eqref{strictposlag1} holds with $P \in \co^{d \times d} $ such that $\im P > 0$. In this case $k=d$,  $\lambda_1= \cc d$ and \eqref{lagrangerep} holds with $Q=2P$. 

Suppose $\lambda \subseteq T^* \cc d$ is a positive but not strictly positive Lagrangian. 
If $k=0$ we have $\lambda=\{0\} \times \cc d$.
If $k > 0$ we define 
\begin{equation*}
\Lambda = \{ (x,ix+(\xi',0)) : \, x \in \cc d, \, \exists \xi'' \in \cc N: \, (x',0,\xi) \in \lambda \} \subseteq T^* \cc d. 
\end{equation*}
To show that the set $\Lambda$ is well defined, suppose that $(x',0,\xi), (x',0,\eta) \in \lambda$. 
Then by linearity $(0,\xi-\eta) \in \lambda$. Since $\lambda$ is Lagrangian we have for any $(y',0,\theta) \in \lambda$
\begin{equation*}
0 = \sigma( (0,\xi-\eta), (y',0,\theta)) = \la \xi'-\eta',y' \ra
\end{equation*}
which implies $\xi' = \eta'$, so $\Lambda$ is indeed well defined. It is also a complex linear space. 

Let $(x,ix+(\xi',0)), (y,iy+(\eta',0)) \in \Lambda$, that is $(x',0,\xi), (y',0,\eta) \in \lambda$ for some $\xi'',\eta'' \in \cc N$. 
Again since $\lambda$ is Lagrangian
\begin{align*}
\sigma( (x,ix+(\xi',0)), (y,iy+(\eta',0)) ) 
& = \la y',\xi' \ra - \la x',\eta' \ra \\
& = \sigma( (x',0,\xi), (y',0,\eta) ) = 0
\end{align*}
which means that 
\begin{equation*}
\Lambda \subseteq  \{ X \in T^* \cc d: \, \sigma(X,Y) =0 \ \forall Y \in \Lambda \}. 
\end{equation*}
Such a subspace is called isotropic (cf. \cite[Definition 21.2.2]{Hormander0}) and has dimension $\leq d$. 
To see that actually $\dim_{\co} \Lambda = d$ denote by $\{ e_j \}_{j=1}^d$ the standard basis vectors in $\cc d$. 
For each $j=1,\dots,d$ there exists $\xi_j \in \cc d$ such that $(e_j',0,\xi_j) \in \lambda$, so $(e_j, i e_j + (\xi_j',0) ) \in \Lambda$ for each  
$j=1,\dots,d$. The set 
\begin{equation*}
\{ (e_j, i e_j + (\xi_j',0) ) \}_{j=1}^d \subseteq T^* \cc d
\end{equation*}
is clearly linearly independent over $\co$, so $\dim_{\co} \Lambda \geq d$. Thus we have $\dim_{\co} \Lambda = d$. 
According to \cite[Chapter 21.2]{Hormander0} $\Lambda$ is then Lagrangian. 

Let $(x,ix+(\xi',0)) \in \Lambda$ that is $(x',0,\xi) \in \lambda$ for some $\xi'' \in \cc N$.
By the assumption that $\lambda$ is positive we have 
\begin{equation*}
0 \leq i \sigma ( \overline{(x',0,\xi)},  (x',0,\xi) ) = 2 \im \la \overline{x}', \xi' \ra,
\end{equation*}
which gives
\begin{equation}\label{strictposlag2}
i\sigma ( \overline{(x,ix+(\xi',0))},  (x,ix+(\xi',0)) ) = 2 |x|^2 + 2 \im \la \overline{x}', \xi' \ra > 0, \quad x \neq 0. 
\end{equation}
If $x=0$ and $(x,\xi) \in \lambda$ we have for all $(y',0,\eta) \in \lambda$
\begin{equation*}
0 = \sigma( (0,\xi), (y',0,\eta) ) = \la \xi', y' \ra,
\end{equation*}
which implies $\xi'=0$. Thus \eqref{strictposlag2} says that $\Lambda$ is a strictly positive Lagrangian. 

By \cite[Proposition 21.5.9]{Hormander0} there exists a symmetric matrix $A \in \co^{d \times d}$ with $\im A > 0$ such that  
\begin{equation*}
\Lambda = \{ (x,Ax) : \, x \in \cc d \}. 
\end{equation*}
Let $(x',0,\xi) \in \lambda$. Then $( (x',0), i (x',0) + (\xi',0) ) \in \Lambda$ and therefore
\begin{equation*}
i (x',0) + (\xi',0) = A (x',0).
\end{equation*}
It follows that $\xi'=Q x'$ for some symmetric matrix $Q \in \cc {k \times k}$. 
From this argument and $\dim_{\co} \lambda = d$ we obtain
\begin{equation*}
\lambda = \{ (x',0,Qx',\xi''): \, x' \in \cc k, \, \xi'' \in \cc N \}. 
\end{equation*}
The assumed positivity of $\lambda$ is equivalent to $\im Q \geq 0$. 
\end{proof}

In the next result we use the notation $U \perp V$ for linear subspaces $V \subseteq \rr d$ and $U \subseteq \cc d$ to mean that $\re(U), \im (U) \subseteq V^\perp \subseteq \rr d$, that is $\la u,v \ra = 0$ for all $u \in U$ and $v \in V$. 

\begin{prop}\label{poslagrparam}
Let $\lambda \subseteq T^* \cc d$ be a positive Lagrangian
and define the subspace $\lambda_1$ by \eqref{proj1coord}, $k = \dim_{\co} \lambda_1$ and $N=d-k$. 
Then there exists a quadratic form $\rho$ on $\rr d$ with $\im \rho \geq 0$, 
such that if $k=d$ then
\begin{equation}\label{lagrangian2a}
\lambda = \{ (x, \rho'(x)): \,  x \in \cc {d} \}, 
\end{equation}
while if $0 \leq k \leq d-1$ then there exists $L \in \rr {d \times N}$ injective such that 
\begin{equation}\label{orthorange}
\Ran \rho' \perp \Ran L
\end{equation}
and
\begin{equation}\label{lagrangian2}
\lambda = \{ (x, \rho'(x) + L \theta): \,  (x,\theta) \in \cc {d+N}, \, L^t x = 0 \}. 
\end{equation}
\end{prop}

\begin{rem}
Note that $\lambda$ defined by \eqref{lagrangian2a}, \eqref{lagrangian2} respectively, are parametrized as in \eqref{lagrangian1} with $p(x,\theta) = \rho(x)$, 
\begin{equation*}
p(x,\theta) = \rho(x) + \la L \theta,x \ra,
\end{equation*}
respectively. These are quadratic forms on $\rr {d+N}$ (with $N=0$ in the former case) defined by 
by the symmetric Gaussian generator matrices $P=\rho'/2$ and 
\begin{equation}\label{Pdef}
P = \frac1{2} \left(
\begin{array}{ll}
\rho' & L \\
L^t & 0 
\end{array}
\right),
\end{equation}
respectively.
\end{rem}

\begin{proof}
If $k=d$ then Lemma \ref{lagrangianpure} gives the representation \eqref{lagrangian2a} with $\rho(x)=\la x, Q x \ra/2$. 
If $k=0$ then $\lambda= \{0\} \times \cc d$, and then $\rho=0$ and $L=I_d$ satisfy \eqref{orthorange}, \eqref{lagrangian2}.   

We can thus suppose $1 \leq k \leq d-1$ in the rest of the proof.
Let $\{ u_j \}_{j=1}^k \subseteq \rr d$ be an orthonormal basis for $\re \lambda_1$, 
and let  $\{ u_j \}_{j=k+1}^d \subseteq \rr d$ be an orthonormal basis for $(\re \lambda_1)^\perp$. 
Considering the vectors $u_j$ as columns, define the injective matrices
\begin{equation*}
U = (u_1 \cdots u_k) \in \rr {d \times k}, \quad L = (u_{k+1} \cdots u_d) \in \rr {d \times N}, 
\end{equation*}
and $T = (U \ L) \in \rr {d \times d}$. Then $T$ is an orthogonal matrix and 
\begin{equation}\label{zeroproduct}
L^t U = 0 \in \rr {N \times k}. 
\end{equation}
With this construction we have 
\begin{equation*}
\lambda_1 = T (\cc k \times \{0\}). 
\end{equation*}
Define
\begin{equation*}
\lambda_T =  \left(
\begin{array}{ll}
T^t & 0 \\
0 & T^t
\end{array}
\right)
\lambda. 
\end{equation*}
Since 
\begin{equation*}
\left(
\begin{array}{ll}
T^t & 0 \\
0 & T^t
\end{array}
\right) \in \Sp(d,\ro) ,
\end{equation*}
it follows that $\lambda_T$ is a positive Lagrangian. 
Since the projection of $\lambda_T$ on the first $\cc d$ coordinate is $\cc k \times \{0\}$, we may apply Lemma \ref{lagrangianpure} which says that 
\begin{equation*}
\lambda_T = \{ (x',0,Qx',\xi''): \, x' \in \cc k, \, \xi'' \in \cc N \} 
\end{equation*}
for a certain $Q \in \cc {k \times k}$ such that $\im Q \geq 0$. 
Hence 
\begin{align*}
\lambda & =  \left(
\begin{array}{ll}
T & 0 \\
0 & T
\end{array}
\right)
\lambda_T \\
& = \{ (T(x',0),T(Qx',\xi'') ): \, x' \in \cc k, \, \xi'' \in \cc N \}. 
\end{align*}
We have $y=T(x',0)$ for some $x' \in \cc k$ if and only if $T^t y \in \cc k \times \{0\}$ that is $L^t y=0$, 
and $x'=U^t y$. 
Hence
\begin{align*}
\lambda 
& = \{ (x,UQU^t x + L\xi'') ): \, (x,\xi'') \in \cc {d+N}, \, L^t x =0 \}. 
\end{align*}
If we set 
\begin{equation*}
\rho(x) = \frac1{2} \la x, UQU^t x \ra,
\end{equation*}
then $\im \rho \geq 0$ when $\rho$ is considered a quadratic form on $\rr d$, the equality \eqref{lagrangian2} holds, and \eqref{orthorange} follows from \eqref{zeroproduct}. 
\end{proof}

\subsection{The Gabor wave front set of an oscillatory integral}
\label{oscintgaborsec}

The next result is decisive for Section \ref{kernelsec}. It shows that the Gabor wave front set of an oscillatory integral is contained in the intersection of its corresponding Lagrangian with the real subspace $T^* \rr d$. 

\begin{thm}\label{WFinclusion}
Let $u \in \cS'(\rr d)$ be the oscillatory integral \eqref{oscillint1} defined by a quadratic form $p$ and a symmetric Gaussian generator matrix so that its associated positive Lagrangian $\lambda$ is given by \eqref{lagrangian1}. 
Then 
\begin{equation}\label{WFincl}
WF(u) \subseteq (\lambda \cap T^* \rr d) \setminus \{ 0 \}. 
\end{equation}
\end{thm}

\begin{proof}
First we assume that $N \geq 1$ in \eqref{oscillint1} and in the end we will take care of the case $N=0$. 

By \cite[Propositions 5.6 and 5.7]{Hormander2} we may assume that $p$ has the form
\begin{equation*}
p(x,\theta) = \rho(x) + \la L \theta,x \ra
\end{equation*}
where $\rho$ is a quadratic form, $\im \rho \geq 0$ and $L \in \rr {d \times N}$ is injective. 
The term of $p$ that is quadratic in $\theta$ can thus be eliminated. 
The matrix $L$ is uniquely determined by $\lambda$ modulo invertible right factors, and the same holds for the values of $\rho$ on $\Ker L^t$. 

The positive Lagrangian \eqref{lagrangian1} associated to $u$ is
\begin{equation}\label{lagrangian3}
\lambda = \{ (x,\rho'(x) + L \theta): \, (x,\theta) \in \cc {d+N}, \, L^t x = 0 \} \subseteq T^* \cc d
\end{equation}
and (cf. \cite[Proposition 5.7]{Hormander2})
\begin{equation*}
u(x) = (2 \pi)^N \delta_0 (L^t x) e^{i \rho(x)}, 
\end{equation*}
where $\delta_0 = \delta_0(\rr N)$. 
By Proposition \ref{poslagrparam} and the uniqueness part of \cite[Proposition 5.7]{Hormander2} we may assume 
\begin{equation*}
\Ran \rho' \perp \Ran L. 
\end{equation*}
From this it follows that if  $x \in \rr d$, $\theta \in \cc N$ and $0 = \im (\rho'(x) + L \theta) = \im (\rho')(x) + L \im \theta$, 
then $\im (\rho')(x)=0$ and $L \im \theta=0$, so the injectivity of $L$ forces $\theta \in \rr N$. 
Thus 
\begin{equation*}
\lambda \cap T^* \rr d = \{ (x,\re(\rho')(x) + L \theta): \, (x,\theta) \in \rr {d+N}, \, L^t x = 0 \}. 
\end{equation*}
According to \cite[Proposition 2.3]{Hormander1} we have
\begin{equation*}
WF( \delta_0 (L^t \cdot) ) 
= \{ (x,L \xi) \in T^* \rr d \setminus 0 : \, (L^t x,\xi) \in WF(\delta_0) \} \cup \Ker L^t \setminus 0  \times \{ 0 \},
\end{equation*}
where $\Ker L^t = \Ker L^t \cap \rr d$ is understood. 
By \eqref{example1} $WF(\delta_0)= \{ 0 \} \times (\rr N \setminus \{0\})$ so 
\begin{align*}
WF( \delta_0 (L^t \cdot) ) 
& = \{ (x,L \xi) \in T^* \rr d \setminus 0: \, L^t x=0, \, \xi \in \rr N \setminus 0 \} \cup \Ker L^t \setminus 0 \times \{ 0 \} \\
& = (\Ker L^t \times L \rr N) \setminus 0. 
\end{align*}
We write $\rho(x) = \rho_r(x) + i \rho_i(x)$ and 
\begin{equation}\label{rhodef}
\rho_r (x) = \la x, R_r x \ra, \quad \rho_i (x) = \la x, R_i x \ra,
\end{equation}
with $R_r, R_i \in \rr {d \times d}$ symmetric and $R_i \geq 0$. 
The function $e^{ i \rho_r(x)}$, considered as a multiplication operator, is the metaplectic operator (cf. \eqref{symplecticoperator} and \cite{Folland1})
\begin{equation*}
e^{ i \rho_r(x)} = \mu (\chi),
\end{equation*}
where 
\begin{equation}\label{chidef}
\chi = 
\left(
\begin{array}{ll}
I_d & 0 \\
2 R_r & I_d 
\end{array}
\right) \in \Sp(d,\ro).
\end{equation}
Since 
 $R_i \geq 0$ we have 
\begin{equation*}
e^{- \rho_i(x) } \in S_{0,0}^0,
\end{equation*}
when we consider $e^{- \rho_i(x) }$ as a function of $(x,\xi) \in T^* \rr d$, constant with respect to the $\xi$ variable. 
The corresponding Weyl pseudodifferential operator, which is a multiplication operator, is by \eqref{microlocal2} microlocal with respect to the Gabor wave front set: 
\begin{equation}\label{propagation0}
WF(e^{- \rho_i } u) \subseteq WF(u), \quad u \in \cS'(\rr d). 
\end{equation}
Piecing these arguments together, using \eqref{symplecticinvarianceWF}, gives
\begin{equation}\label{WFincl0}
\begin{aligned}
WF( u ) 
& = WF ( e^{- \rho_i } e^{i \rho_r} \delta_0 (L^t \cdot) ) \\
& \subseteq WF ( e^{i \rho_r} \delta_0 (L^t \cdot) ) \\
& = \chi WF ( \delta_0 (L^t \cdot) ) \\
& = \{ (x, 2 R_r x + L \theta):  \, (x,\theta)  \in \rr {d+N}, \, L^t x = 0 \} \setminus \{ 0 \} \\
& = \{ (x, \rho_r'(x) + L \theta):  \, (x,\theta)  \in \rr {d+N}, \, L^t x = 0 \} \setminus \{ 0 \} \\
& = (\lambda \cap T^* \rr d) \setminus \{ 0 \}. 
\end{aligned}
\end{equation}
This ends the proof when $N \geq 1$. 
Finally, if $N=0$ then the oscillatory integral \eqref{oscillint1} is interpreted as
\begin{equation*}
u(x) = e^{ i \rho(x)},
\end{equation*} 
where $\rho = \rho_r + i \rho_i$ is given by \eqref{rhodef}
with $R_r, R_i \in \rr {d \times d}$ symmetric and $R_i \geq 0$. 
By \eqref{lagrangian1} the corresponding positive Lagrangian is 
\begin{equation}\label{lambda0}
\lambda = \{ (x, 2 (R_r + i R_i) x): \, x \in \cc d \} \subseteq T^* \cc d,
\end{equation}  
which gives 
\begin{equation*}
\lambda \cap T^* \rr d = \{ (x, 2 R_r x): \, x \in \rr d \}. 
\end{equation*}  
Since $WF(1)= (\rr d \setminus 0) \times \{0\}$ we obtain, 
again using \eqref{microlocal2} and recycling the argument above, 
\begin{equation*}
\begin{aligned}
WF( u ) 
& = WF ( e^{- \rho_i } e^{i \rho_r} ) \\
& \subseteq WF ( e^{i \rho_r}  ) \\
& = \chi WF ( 1 ) \\
& = \{ (x, 2 R_r x):  \, x  \in \rr d \setminus 0 \}  \\
& = (\lambda \cap T^* \rr d) \setminus \{ 0 \}. 
\end{aligned}
\end{equation*}
\end{proof}

\begin{rem}\label{remark1}
It follows from the proof that 
\begin{equation*}
\lambda \cap T^* \rr d = \widetilde \lambda \cap T^* \rr d,
\end{equation*}
where 
\begin{equation}\label{tildelambda}
\widetilde \lambda = \{ (x,\xi) \in \lambda: \, R_i x = 0 \} \subseteq \lambda. 
\end{equation}
Therefore \eqref{WFincl} could be replaced by 
\begin{equation}\label{WFincl1}
WF(u) \subseteq (\widetilde \lambda \cap T^* \rr d) \setminus \{ 0 \}, 
\end{equation}
in the conclusion of Theorem \ref{WFinclusion}, that is
it suffices to work with the smaller space $\widetilde \lambda$ instead of $\lambda$. 
We do not exclude that in particular situations it may be an advantage to deal with the smaller space 
$\widetilde \lambda$ instead of $\lambda$. 

We can characterize the subspace $\widetilde \lambda \subseteq T^* \cc d$ intrinsically as follows. 
Suppose $N \geq 1$. 
Then \eqref{lagrangian3} holds. 
Since $\lambda$ is a positive Lagrangian we have for any $X = (x,\rho'(x) + L \theta) \in \lambda$, 
since $L^t x = 0$,
\begin{align*}
0 & \leq i \sigma(\overline{(x,\rho'(x) + L \theta)},(x,\rho'(x) + L \theta)) \\
& = i \left( \la \overline{\rho'(x)}, x \ra + \la \overline{\theta}, L^t x \ra - \la \overline{x}, \rho'(x) \ra - \la \overline{L^t x}, \theta \ra \right) \\
& = i \left( \la \overline{\rho'(x)}, x \ra - \la \overline{x}, \rho'(x) \ra \right) \\
& = 2 \, \im \la  \rho'(x), \overline{x} \ra \\
& = 4 \, \la  R_i x, \overline{x} \ra. 
\end{align*} 
It follows that $R_i x = 0$ if and only if $\sigma(\overline X,X)=0$.  
In view of \eqref{tildelambda} we thus obtain the following criterion of when $X \in \lambda$ belongs to the subspace $\widetilde \lambda \subseteq \lambda$. 
\begin{equation}\label{kritlamdbatilde}
\mbox{If} \ \, X \in \lambda \ \, \mbox{then} \ X \in \widetilde \lambda \quad \Longleftrightarrow \quad \sigma(\overline X,X)=0. 
\end{equation}
We note that this criterion extends also to the case when $N=0$, i.e. when $\lambda$ is given by \eqref{lambda0}.  
\end{rem}

\section{The kernel of the solution operator and its Gabor wave front set}
\label{kernelsec}

In this section we will use H\"ormander's result \cite[Theorem 5.12]{Hormander2} where the Schwartz kernel of the propagator 
$e^{-t q^w(x,D)}$ is shown to be an oscillatory integral associated to the twisted positive Lagrangian given by the graph of a matrix in $\Sp(d,\co)$. 
Combined with Theorem \ref{WFinclusion} this will give an inclusion for the Gabor wave front set of the Schwartz kernel of $e^{-t q^w(x,D)}$, that will be useful in Section \ref{generalcasesec}.

In order to explain the term twisted Lagrangian we first have to look at Lagrangians in the product of symplectic spaces $T^* \cc d \times T^* \cc d$. 

Consider the symplectic form defined on $T^* \cc d \times T^* \cc d$ by
\begin{equation}\label{sympform1}
\begin{aligned}
\sigma_1(X, Y) & = \sigma(x_1,y_1) - \sigma(x_2,y_2), \quad X=(x_1,x_2) \in T^* \cc d \times T^* \cc d , \\
& \qquad \qquad \qquad \qquad \qquad \quad \ \, Y=(y_1,y_2) \in T^* \cc d \times T^* \cc d, 
\end{aligned}
\end{equation}
where $\sigma$ is the symplectic form \eqref{cansympform} for $T^* \cc d$. 
Let $\chi \in \Sp(d,\co)$. 
Since
\begin{equation*}
0 = \sigma(\chi X, \chi Y) - \sigma(X,Y) = \sigma_1( (\chi X,X); (\chi Y, Y) ), \quad X,Y \in T^* \cc d, 
\end{equation*}
the graph of $\chi$
\begin{equation*}
\mathcal G(\chi) = \{ (\chi X,X): \, X \in T^* \cc d \},
\end{equation*}
is a Lagrangian subspace in $T^* \cc d \times T^* \cc d$ with respect to the symplectic form $\sigma_1$. 

In the next result we consider three symplectic vector spaces: 

\begin{enumerate}

\item $T^* \cc d \times T^* \cc d$ equipped with the symplectic form $\sigma_1$; 

\item $T^* \cc d \times T^* \cc d$ equipped with the symplectic form  
\begin{equation*}
\begin{aligned}
\sigma_2(X, Y) & = \sigma(x_1,y_1) + \sigma(x_2,y_2), \quad X=(x_1,x_2) \in T^* \cc d \times T^* \cc d , \\
& \qquad \qquad \qquad \qquad \qquad \quad \ \, Y=(y_1,y_2) \in T^* \cc d \times T^* \cc d;
\end{aligned}
\end{equation*}

\item $T^* \cc {2d}$ equipped with its canonical symplectic form \eqref{cansympform}.

\end{enumerate}

We use the notation $(x,\xi)'=(x,-\xi)$ for $(x,\xi) \in T^* \cc d$. 

\begin{lem}\label{sympekv}
Let $\lambda \subseteq T^* \cc d \times T^* \cc d$ be a complex linear subspace. Then 
$\lambda$ is a Lagrangian subspace with respect to $\sigma_1$ if and only if 
\begin{equation}\label{twistdef}
\lambda' = \{ (X,Y') \in T^* \cc d \times T^* \cc d : \, (X,Y) \in \lambda \},
\end{equation}
is a Lagrangian subspace with respect to $\sigma_2$, if and only if 
\begin{equation*}
\widetilde \lambda' = \{ (x_1,x_2; \xi_1, \xi_2) \in T^* \cc {2d} : \, (x_1, \xi_1; x_2, \xi_2) \in \lambda' \},
\end{equation*}
is a Lagrangian subspace of $T^* \cc {2d}$. 
The transformations preserves positivity of Lagrangians. 
\end{lem}

\begin{proof}
Let $X =(x_1,x_2) \in T^* \cc d \times T^* \cc d$ with $x_j = (x_{j 1},x_{j 2}) \in T^* \cc d$, $j=1,2$, and likewise for $Y \in T^* \cc d \times T^* \cc d$. Then 
\begin{align*}
\sigma_1 (X,Y) 
& = \sigma(x_1,y_1) - \sigma(x_2,y_2) \\
& = \la x_{1 2}, y_{1 1} \ra - \la x_{1 1}, y_{1 2} \ra
- \la x_{2 2}, y_{2 1} \ra + \la x_{2 1}, y_{2 2} \ra \\
& = \la x_{1 2}, y_{1 1} \ra - \la x_{1 1}, y_{1 2} \ra
+ (\la - x_{2 2}, y_{2 1} \ra - \la x_{2 1}, -y_{2 2} \ra) \\
& = \sigma(x_1,y_1) + \sigma(x_2',y_2') \\
& = \sigma_2((x_1,x_2'), (y_1,y_2')),  
\end{align*}
which proves the first equivalence. 

To show the second equivalence let $X, Y \in T^* \cc d \times T^* \cc d$ with $X=(x_{1 1}, x_{1 2}, x_{2 1}, x_{2 2})$, 
$\widetilde X = (\widetilde X_1,\widetilde X_2) = (x_{1 1}, x_{2 1}, x_{1 2}, x_{2 2}) \in T^* \cc {2d}$, $x_j = (x_{j 1},x_{j 2}) \in T^* \cc d$, $j=1,2$, and likewise for $Y$. We have, using the fourth and fifth equality above,
\begin{align*}
\sigma (\widetilde X,\widetilde Y) 
& = \la \widetilde X_2, \widetilde Y_1 \ra - \la \widetilde X_1, \widetilde Y_2 \ra \\
& = \la x_{1 2}, y_{1 1} \ra +  \la x_{2 2}, y_{2 1} \ra - \la x_{1 1}, y_{1 2} \ra - \la x_{2 1}, y_{2 2} \ra \\
& = \sigma_2((x_1,x_2), (y_1,y_2)),
\end{align*}
which proves the second equivalence. 
The transformations clearly preserve positivity. 
\end{proof}

In the literature Lagrangians in $T^* \cc {2d}$ with its canonical symplectic form \eqref{cansympform} are often identified with Lagrangians in $T^* \cc d \times T^* \cc d$ with the symplectic form $\sigma_2$,
due to the second equivalence in Lemma \ref{sympekv}. We adopt this convention from now on. 

By the first equivalence of Lemma \ref{sympekv} each positive Lagrangian $\lambda$ in $T^* \cc {2d}$  corresponds to a twisted positive Lagrangian $\lambda'$, defined by a change of sign in the fourth coordinate as in \eqref{twistdef}, in 
$T^* \cc d \times T^* \cc d$ with the form $\sigma_1$. 

For a quadratic form $q$ on $T^*\rr d$ defined by a symmetric matrix $Q \in \cc {2 d \times 2 d}$ with Hamiltonian $F=\J Q$ and $\re Q \geq 0$,
we show in the next lemma that the matrix $e^{-2 i t F} \in \cc {2d \times 2d}$ for $t \geq 0$ belongs to $\Sp(d,\co)$, and its graph in $T^* \cc d \times T^* \cc d$ is a positive Lagrangian with respect to the symplectic form $\sigma_1$.

\begin{lem}\label{poslemma}
Let $q$ be a quadratic form on $T^*\rr d$ defined by a symmetric matrix $Q \in \cc {2 d \times 2 d}$ with Hamiltonian $F=\J Q$ and $\re Q \geq 0$. 
Then for any $t \geq 0$
\begin{equation*}
e^{-2 i t F} \in \Sp(d,\co)
\end{equation*}
and the graph $\mathcal G( e^{-2 i t F}  )$ is a positive Lagrangian in $T^* \cc d \times T^* \cc d$ with respect to the symplectic form $\sigma_1$. 
\end{lem}

\begin{proof}
Using $F=\J Q$, $\J^{-1}=-\J= \J^t$ and the symmetry of $Q$ we have
$$
\J F^t \J = \J ( -\J F)^t = \J Q^t = F,
$$
which gives
$$
- \J e^{-2 i t F^t} \J = e^{2 i t \J F^t \J} = e^{2 i t F}. 
$$
This gives for any $X,Y \in T^* \cc d$
\begin{align*}
\sigma( e^{-2 i t F} X, e^{- 2 i t F} Y ) 
& = \la \J e^{-2 i t F} X, e^{- 2 i t F} Y \ra 
= \la e^{-2 i t F^t} \J e^{-2 i t F} X, Y \ra \\
& = \la \J (-\J) e^{-2 i t F^t} \J e^{-2 i t F} X, Y \ra \\
& = \la \J e^{2 i t F} e^{-2 i t F} X, Y \ra
= \la \J  X, Y \ra
= \sigma(X,Y), 
\end{align*}
which proves $e^{-2 i t F} \in \Sp(d,\co)$.

To show that the graph $\mathcal G( e^{-2 i t F}  )$ is a positive Lagrangian in $T^* \cc d \times T^* \cc d$ with respect to the symplectic form $\sigma_1$ we must show that 
\begin{equation*}
f(t) := i \left( \sigma( \overline{e^{-2 i t F} X}, e^{- 2 i t F} X) - \sigma(\overline X,X)  \right) \geq 0,
\end{equation*}
for any $X \in T^* \cc d$ and any $t \geq 0$. Fix $X \in T^* \cc d$. 
The function $f$ is real-valued due to the antisymmetry of the symplectic form $\sigma$. 
Since $f(0)=0$ the nonnegativity of $f$ will follow if we show $f'(t) \geq 0$ for any $t \geq 0$. 
We have
\begin{align*}
F^t \J - \J \overline F
= (-\J F)^t  - \J \overline F 
= Q - \J^2 \overline Q = Q + \overline Q = 2  \, \re Q. 
\end{align*}
This gives for any $t \geq 0$
\begin{align*}
f'(t) & = -2 \left( \sigma( \overline{F \, e^{-2 i t F} X}, e^{-2 i t F} X) - \sigma( \overline{e^{-2 i t F} X}, F e^{-2 i t F} X)  \right) \\
& = 2 \la ( F^t \J - \J \overline{F} ) \, \overline{e^{-2 i t F} X}, e^{-2 i t F} X \ra \\
& = 4 \la \re Q \, \overline{e^{-2 i t F} X}, e^{- 2 i t F} X \ra \geq 0.
\end{align*}
\end{proof}

We will now bring together the results above with H\"ormander's results \cite{Hormander2}. 
Let $q$ be a quadratic form on $T^*\rr d$ defined by a symmetric matrix $Q \in \cc {2 d \times 2 d}$ with Hamiltonian $F=\J Q$ and $\re Q \geq 0$.
According to \cite[Theorem 5.12]{Hormander2} the Schwartz kernel of the propagator $e^{-t q^w(x,D)}$ for $t \geq 0$ is an oscillatory integral of type \eqref{oscillint1}.   
More precisely we have
\begin{equation*}
e^{-t q^w(x,D)} = \cK_{e^{-2 i t F}},
\end{equation*}
where $\cK_{e^{-2 i t F}}: \cS(\rr d) \mapsto \cS'(\rr d) $ is the linear continuous operator with Schwartz kernel 
\begin{multline}\label{schwartzkernel1}
K_{e^{-2 i t F}} (x,y) \\
= (2 \pi)^{-(d+N)/2} \sqrt{\det \left( 
\begin{array}{ll}
p_{\theta \theta}''/i & p_{\theta y}'' \\
p_{x \theta}'' & i p_{x y}'' 
\end{array}
\right) } \int e^{i p(x,y,\theta)} d\theta \in \cS'(\rr {2d}),  
\end{multline}
where the quadratic form $p$ will be specified shortly. 

By \cite[Proposition~5.8]{Hormander2} $\cK_{e^{-2 i t F}}$ is in fact continuous on $\cS (\rr d)$.
The kernel $K_{e^{-2 i t F}}$ is indexed by the matrix $e^{-2 i t F}$. By Lemma \ref{poslemma}, the matrix $e^{-2 i t F}$ belongs to $\Sp(d,\co)$ and its graph 
\begin{equation}\label{graph1}
\lambda' : = \mathcal G(e^{- 2 i t F}) \subseteq T^* \cc d \times T^* \cc d,
\end{equation}
is a positive Lagrangian with respect to the symplectic form $\sigma_1$ defined by \eqref{sympform1}. 
As explained after Lemma \ref{sympekv} the Lagrangian $\lambda'$ can be twisted as in \eqref{twistdef} to give a positive Lagrangian $\lambda \subseteq T^* \cc {2d}$. 
According to \cite[Theorem 5.12 and p.~444]{Hormander2} the oscillatory integral \eqref{schwartzkernel1} is associated with the positive Lagrangian $\lambda$. 

By Proposition \ref{poslagrparam} there exists a quadratic form $p$ on $\rr {2d + N}$ defined by a symmetric Gaussian generator matrix $P \in \cc {(2d+N) \times (2d+N)}$ (cf. \eqref{Pdef}) that defines $\lambda$ by means of \eqref{lagrangian1} with $x \in \cc d$ replaced by $(x,y) \in \cc {2d}$. 
This form $p$ defines \eqref{schwartzkernel1}. 

The factor in front of the integral  \eqref{schwartzkernel1} is designed to make the oscillatory integral independent of the quadratic form $p$ on $\rr {2d + N}$, including possible changes of dimension $N$ as discussed after Proposition \ref{oscintprop}, as long as $p$ defines $\lambda$ by means of \eqref{lagrangian1} with $x \in \cc d$ replaced by $(x,y) \in \cc {2d}$. 

It is shown in \cite[p.~444]{Hormander2} that the kernel $K_{e^{-2 i t F}}$ is uniquely determined by the Lagrangian $\lambda$, apart from a sign ambiguity which is not essential for our purposes. 

If we now combine these considerations with Theorem \ref{WFinclusion} we obtain the following inclusion for the Gabor wave front set of the Schwartz kernel $K_{e^{- 2 i t F}}$ of the propagator $e^{- t q^w(x,D)}$. 

\begin{thm}\label{WFkernel}
Let $q$ be a quadratic form on $T^*\rr d$ defined by a symmetric matrix $Q \in \cc {2 d \times 2 d}$ with Hamiltonian $F=\J Q$ and $\re Q \geq 0$.
Consider the Schwartz kernel \eqref{schwartzkernel1} of the propagator $e^{-t q^w(x,D)}$ for $t \geq 0$. 
Then
\begin{equation}\label{WFkernel1}
\begin{aligned}
& WF( K_{e^{-2 i t F}} ) \\
& \subseteq \{ (x, y, \xi, -\eta) \in T^* \rr {2d} \setminus 0 : \, (x,\xi) =  e^{-2 i t F} (y,\eta), \, \im e^{- 2 i t F} (y,\eta) = 0 \}. 
\end{aligned}
\end{equation}
\end{thm}

\begin{proof}
By the discussion above, \eqref{graph1} and Lemma \ref{sympekv} the Lagrangian $\lambda  \subseteq T^* \cc {2d}$ associated to the oscillatory integral \eqref{schwartzkernel1} is the twisted graph
\begin{equation}\label{twistgraph}
\lambda = \{ (x, y, \xi, -\eta) \in T^* \cc {2d}: \, (x,\xi) =  e^{-2 i t F} (y,\eta) \}. 
\end{equation}
The result now follows from Theorem \ref{WFinclusion}. 
\end{proof}

\begin{rem}\label{remark2}
As we saw in Remark \ref{remark1} we could replace the positive Lagrangian $\lambda$ by the subspace $\widetilde \lambda \subseteq \lambda$ in the use of Theorem \ref{WFinclusion}. 
For the positive Lagrangian $\lambda'$, defined by the graph \eqref{graph1} in $T^* \cc d \times T^* \cc d \simeq T^* \cc {2d}$ of the matrix $e^{-2 i t F}$, equipped with the symplectic form $\sigma_1$ defined by \eqref{sympform1}, the criterion $\sigma(\overline X,X)=0$ in \eqref{kritlamdbatilde} is
\begin{equation}\label{krittildebidim}
\sigma( \overline{e^{- 2 i t F} (y,\eta)}, e^{- 2 i t F} (y,\eta)) = \sigma( \overline{(y,\eta)}, (y,\eta)). 
\end{equation}
In view of \eqref{kritlamdbatilde} this gives for $\lambda$ defined by \eqref{twistgraph}
\begin{equation}\label{tildelambdaK}
\begin{aligned}
\widetilde \lambda & = \{ (x,y,\xi,-\eta) \in T^* \cc {2d} : \, (x,\xi) = e^{-2 i t F} (y,\eta), \\
& \qquad \qquad \quad \sigma(\overline{(x,\xi)},(x,\xi)) = \sigma(\overline{(y,\eta)},(y,\eta) ) \}
\end{aligned}
\end{equation}
and \eqref{WFkernel1} can hence be alternatively formulated
\begin{equation}\label{WFkernel1a}
WF( K_{e^{-2 i t F}} ) \subseteq  (\widetilde \lambda \cap T^* \rr {2d}) \setminus \{ 0 \}. 
\end{equation}
\end{rem}

\section{Propagation of Gabor singularities in the singular space}
\label{generalcasesec}

In this section, we prove our general result on propagation of Gabor singularities in the singular space. 
This is achieved using Theorem \ref{WFkernel}. 
Secondly, we sort out the special case when the Poisson bracket $\{q, \overline q\} = 0$ is zero.

In the following lemma we use the relation mapping between two sets $A \subseteq X_1 \times X_2$ and $B \subseteq X_2$,
\begin{equation*}
A \circ B = \{ X \in X_1 : \, \exists Y \in B : \, (X,Y) \in A \} \subseteq X_1. 
\end{equation*}

\begin{lem}\label{propagationsing2}
Let $q$ be a quadratic form on $T^*\rr d$ defined by a symmetric matrix $Q \in \cc {2d \times 2d}$, $F=\J Q$ and $\re Q \geq 0$. 
Suppose $K_{e^{-2 it F}}$ is the oscillatory integral \eqref{schwartzkernel1} associated to the positive Lagrangian $\lambda$ whose twisted positive Lagrangian $\lambda'$ is \eqref{graph1}.
Then for $u \in \cS'(\rr d)$
\begin{align*}
WF(e^{- t q^w(x,D)}u) 
& \subseteq  WF( K_{e^{-2 i t F}} )' \circ WF(u) \\
& \subseteq  e^{-2 i t F} \left( WF(u) \cap \Ker (\im e^{-2 i t F} ) \right), \quad t \geq 0. 
\end{align*}
\end{lem}

\begin{proof}
Consider the Schwartz kernel $K_{e^{-2 i t F}}$ of the propagator $e^{-t q^w(x,D)}$. 
Since $e^{-2 i t F} \in \cc {2d \times 2d}$ is an invertible matrix, it follows from Theorem \ref{WFkernel} that the wave front set
$WF( K_{e^{-2 i t F}} )$ contains no points of the form $(0,y,0,-\eta)$ for $(y,\eta) \in T^* \rr d \setminus 0$, nor points of the form $(x,0,\xi,0)$ for $(x,\xi) \in T^* \rr d \setminus 0$. 
Thus \cite[Proposition 2.11]{Hormander1} gives the microlocal inclusion
\begin{equation}\label{composition1}
WF(e^{-t q^w(x,D)}u) 
\subseteq  WF( K_{e^{-2 i t F}} )' \circ WF(u), 
\end{equation}
expressed using the twisted Gabor wave front set of $K_{e^{-2 i t F}}$. 
From this we obtain using Theorem \ref{WFkernel}
\begin{align*}
& WF(e^{-t q^w(x,D)}u) \\
& \subseteq \{ (x,\xi) \in T^* \rr d: \, \exists (y,\eta) \in WF(u): \,  (x,y,\xi,-\eta) \in WF( K_{e^{-2 i t F}})  \} \\
& \subseteq e^{-2 i t F} \left( WF(u) \cap \Ker (\im e^{-2 i t F} ) \right),
\end{align*}
where as usual we understand $\Ker (\im e^{-2 i t F} )= \Ker (\im e^{-2 i t F}) \cap T^* \rr d$. 
\end{proof}

In what remains of this section, we prove a sharpening of Lemma \ref{propagationsing2} that is expressed using the notion of singular space. To that end, we notice from the definition (\ref{h1}) that the singular space is also equal to the following infinite intersection of kernels
\begin{equation}\label{h1g}
S=\Big(\bigcap_{j=0}^{+\infty} \Ker \big[\re F(\im F)^j \big]\Big) \cap T^*\rr d \subseteq T^*\rr d, 
\end{equation}
since the Cayley--Hamilton theorem implies that 
$$
(\im F)^k X \in \textrm{Span}(X,...,(\im F)^{2d-1}X),
$$
for all $X \in T^*\rr d$ and $k \geq 0$.
From the description (\ref{h1g}), we observe that the singular space enjoys the following stability properties
\begin{equation}\label{ee2}
(\re F) S=\{0\}, \quad (\im F)S \subseteq S. 
\end{equation}
We are now in a position to state and prove the announced sharpening of Lemma~\ref{propagationsing2} that is our general result on propagation of Gabor singularities. This result says that the singularities 
actually only propagate inside the singular space $S$, following the integral curves of the Hamilton vector field $H_{\textrm{Im}q}$ associated to the imaginary part of the Weyl symbol.

\medskip

\begin{thm}\label{th1}
Let $q$ be a quadratic form on $T^*\rr d$ defined by a symmetric matrix $Q \in \cc {2d \times 2d}$ satisfying $\re Q \geq 0$, and $F = \J Q$ its Hamilton map. 
Then the following microlocal inclusion holds for all $u \in \mathscr{S}'(\rr d)$, $t > 0$,
\begin{align*}
WF\big(e^{-tq^w(x,D)}u\big) \subseteq \left( \exp(tH_{\emph{\textrm{Im}}q})\big(WF(u) \cap S\big) \right) \cap S, 
\end{align*}
where $S$ denotes the singular space of $q$ and where $\exp(tH_{\emph{\textrm{Im}}q})=e^{2t \im F}$ stands for the flow of 
the Hamilton vector field $H_{\emph{\textrm{Im}}q}$ associated to $\im q$.
\end{thm}

\medskip

\begin{proof}
By Lemma \ref{propagationsing2} we have 
\begin{equation}\label{we1}
WF\big(e^{-t q^w(x,D)}u\big) \subseteq e^{-2itF}\big(WF(u) \cap \Ker (\im e^{-2itF})\big),
\end{equation}
for $u \in \mathscr{S}'(\rr d)$ and $t \geq 0$, 
with $\im e^{-2itF}=(2i)^{-1}(e^{-2itF}-e^{2it\overline{F}})$. 
Since $WF(u) \subseteq T^*\rr d$, we observe that 
\begin{align}\label{we0}
 & \ e^{-2itF}\big(WF(u) \cap \Ker (\im e^{-2itF})\big)\\ \notag
= & \ e^{-2itF}\big(WF(u) \cap \Ker (\im e^{-2itF}) \cap T^*\rr d \big)\\ \notag
= & \ \big(e^{-2itF}WF(u)\big) \cap \Big(e^{-2itF}\big[\Ker (\im e^{-2itF}) \cap T^*\rr d \big]\Big)\\ \notag
= & \ \big(e^{-2itF}WF(u)\big) \cap \Ker (\im e^{2itF}) \cap T^*\rr d,
\end{align}
because 
$$
X \in \Ker (\im e^{2itF}) \cap T^*\rr d \quad \Longleftrightarrow \quad Y \in \Ker (\im e^{-2itF}) \cap T^*\rr d,$$
when $X=e^{-2itF}Y$.
It follows from \eqref{we0} that \eqref{we1} may be rephrased as
\begin{equation}\label{we-1}
WF\big(e^{-tq^w(x,D)}u\big) \subseteq
\big(e^{-2itF}WF(u)\big) \cap \Ker (\im e^{2itF}) \cap T^*\rr d, 
\end{equation}
for all $u \in \mathscr{S}'(\rr d)$ and $t \geq 0$. 
We deduce from \cite[Propositions~5.8 and~5.9]{Hormander2} and Lemma \ref{poslemma} that for all $t \geq 0$, the linear continuous operator 
$$\mathscr{K}_{e^{-2itF}} : \mathscr{S}(\rr d) \rightarrow \mathscr{S}'(\rr d),$$ is a continuous map on $\mathscr{S}(\rr d)$, and extends to a continuous map in $\mathscr{S}'(\rr d)$, satisfying
\begin{equation}\label{sd2}
\mathscr{K}_{e^{-2i(t_1+t_2)F}}=\pm \mathscr{K}_{e^{-2it_1F}}\mathscr{K}_{e^{-2it_2F}},
\end{equation}
for all $t_1,t_2 \geq 0$. It follows from \cite[Theorem~5.12]{Hormander2} and \eqref{sd2} that 
\begin{equation}\label{sd3}
e^{-(t_1+t_2)q^w(x,D)}=\pm e^{-t_1q^w(x,D)}e^{-t_2q^w(x,D)},
\end{equation}
for all $t_1,t_2 \geq 0$. 
We deduce from \eqref{we-1} and \eqref{sd3} that for all $u \in \cS'(\rr d)$ and $t=t_1+t_2 \geq 0$, with $t_1,t_2 \geq 0$,
\begin{align}\label{we3}
& \ WF\big(e^{-tq^w(x,D)}u\big)=WF\big(e^{-t_1q^w(x,D)}e^{-t_2q^w(x,D)}u\big) \\ \notag
\subseteq & \ \big[e^{-2it_1F}WF\big(e^{-t_2q^w(x,D)}u\big)\big] \cap \Ker (\im e^{2it_1F}) \cap T^*\rr d \\ \notag
\subseteq & \ \Big[e^{-2it_1F}\Big(\big(e^{-2it_2F}WF(u)\big) \cap \Ker (\im e^{2it_2F}) \cap T^*\rr d \Big)\Big] \cap \Ker (\im e^{2it_1F}) \cap T^*\rr d \\ \notag
\subseteq & \ \big(e^{-2itF}WF(u)\big) \cap \big[e^{-2it_1F}\big(\Ker (\im e^{2it_2F}) \cap T^*\rr d \big)\big] \cap \Ker (\im e^{2it_1F}) \cap T^*\rr d \\ \notag
\subseteq & \ \big(e^{-2itF}WF(u)\big) \cap \Ker (\im e^{2it_1F}) \cap T^*\rr d.
\end{align}
We deduce from \eqref{we3} that for all $t \geq 0$,
\begin{equation}\label{we7}
WF\big(e^{-tq^w(x,D)}u\big) \subseteq \big(e^{-2itF}WF(u)\big) \cap \Big(\bigcap_{0 \leq s \leq t} \Ker (\im e^{2isF})\Big)\cap T^*\rr d.
\end{equation}
We aim at establishing the following formula for the singular space
\begin{equation}\label{qw1}
\forall t>0, \quad S=\Big(\bigcap_{0 \leq s \leq t} \Ker (\im e^{2isF})\Big)\cap T^*\rr d.
\end{equation}
Let $t>0$. We notice that
$$
X \in \Big(\bigcap_{0 \leq s \leq t} \Ker (\im e^{2isF})\Big)\cap T^*\rr d,
$$
if and only if $X \in T^*\rr d$ is such that the function 
$$f(s)=(\im e^{2isF})X=\im (e^{2isF}X)=\im \Big(\sum_{k=0}^{+\infty}\frac{(2isF)^kX}{k!}\Big)=\sum_{k=1}^{+\infty}\frac{2^k\textrm{Im}\big((iF)^k\big)X}{k!}s^k,$$
is zero on the interval $[0,t]$. It implies that for all $t>0$,
\begin{equation}\label{sd6}
\Big(\bigcap_{0 \leq s \leq t} \Ker (\im e^{2isF})\Big)\cap T^*\rr d= \Big(\bigcap_{k= 1}^{+\infty} \Ker \big[\im \big((iF)^k\big)\big]\Big)\cap T^*\rr d.
\end{equation}

Next we check by induction that 
\begin{equation}\label{basis}
\Big(\bigcap_{k=1}^m \Ker \big[\im \big((iF)^k\big)\big]\Big)\cap T^*\rr d
=\Big(\bigcap_{k=0}^{m-1} \Ker \big[\re F(\im F)^k\big]\Big)\cap T^*\rr d,
\end{equation}
holds for all integers $m \geq 1$. 
When $m=1$ the formula \eqref{basis} is trivial since $\im (iF) = \re F$. 
Assume that formula \eqref{basis} holds for $m \geq 1$. By expanding the product 
$$
(iF)^{m+1}=(-\im F+i \re F)^{m+1}
$$ 
and taking the imaginary part, we notice that each term in this sum necessarily contains at least one factor $\re F$, so that 
$$
\big[\im \big((iF)^{m+1}\big)\big]X=\big[\im \big((-\im F+i \re F)^{m+1}\big)\big]X=(-1)^m\re F(\im F)^mX,
$$
when 
$$X \in \Big(\bigcap_{k=0}^{m-1} \Ker \big[\re F(\im F)^k\big]\Big)\cap T^*\rr d.$$
This shows that the formula \eqref{basis} holds also for $m+1$. 
By induction  \eqref{basis} holds for any integer $m \geq 1$.

On the other hand, we observe from \eqref{h1} and \eqref{h1g} that 
\begin{equation}\label{cay}
\forall m \geq 2d-1, \quad S=\Big(\bigcap_{k=0}^{m} \Ker \big[\re F(\im F)^k\big]\Big)\cap T^*\rr d.
\end{equation}
We deduce from \eqref{basis} and \eqref{cay} that 
$$
\forall m \geq 2d, \quad S=\Big(\bigcap_{k=1}^m \Ker \big[\textrm{Im}\big((iF)^k\big)\big]\Big)\cap T^*\rr d.
$$
This implies that 
$$
S=\Big(\bigcap_{k=1}^{+\infty} \Ker \big[\im \big((iF)^k\big)\big]\Big)\cap T^*\rr d.
$$
According to \eqref{sd6}, this proves formula \eqref{qw1}. We deduce from \eqref{we7} and \eqref{qw1} that 
\begin{equation}\label{we8}
\forall t > 0, \quad WF\big(e^{-tq^w(x,D)}u\big) \subseteq \big(e^{-2itF}WF(u)\big) \cap S.
\end{equation}
Let $t>0$ and $X \in \big(e^{-2itF}WF(u)\big) \cap S$. We can find $Y \in WF(u) \subseteq T^*\rr d$ such that $X=e^{-2itF}Y \in S \subseteq T^*\rr d$. It follows from (\ref{h1g}) that 
\begin{equation}\label{sd7}
Y=e^{2itF}X=\sum_{k=0}^{+\infty}\frac{(2itF)^kX}{k!}=\sum_{k=0}^{+\infty}\frac{(-2t \im F)^kX}{k!}=e^{-2t \im F}X,
\end{equation}
since $(iF)^kX=(-\im F)^kX$, when $\re F(\im F)^jX=0$ for all $j \geq 0$. We deduce that 
$$
X=e^{2t \im F}Y \in \big(e^{2t \im F}WF(u)\big) \cap S.
$$
Conversely, we observe from \eqref{sd7} that if $X \in \big(e^{2t \im F}WF(u)\big) \cap S$, then 
$$
Y=e^{-2t \im F}X=e^{2itF}X \in WF(u).
$$
This implies that 
\begin{equation}\label{we9}
\begin{aligned}
\forall t > 0, \quad \big(e^{-2itF}WF(u)\big) \cap S 
& = \big(e^{2t \im F}WF(u)\big) \cap S \\
& = e^{2t \im F}\big(WF(u) \cap (e^{-2t \im F}S)\big).
\end{aligned}
\end{equation}
On the other hand, we deduce from \eqref{ee2} that 
\begin{equation*}
e^{-2t \im F}S \subseteq S, 
\end{equation*}
since $S$ is invariant with respect to $\im F$. The linear mapping $e^{-2t \im F}$ is invertible so its restriction to the singular space defines an injective endomorphism of $S$. Since $S$ is a finite-dimensional vector space, we deduce that $e^{-2t \im F}|_S : S \rightarrow S$ is surjective:
\begin{equation}\label{we10}
e^{-2t \im F}S=S.
\end{equation}
It follows from \eqref{we8}, \eqref{we9}
and \eqref{we10} that 
\begin{equation}\label{we11}
\forall t > 0, \quad WF\big(e^{-tq^w(x,D)}u\big) \subseteq \left( e^{2t \im F} \left( WF(u) \cap S \right) \right) \cap S.
\end{equation}
We observe that the Hamilton vector field associated to the imaginary part of the quadratic symbol $q$ is equal to 
$$
H_{\im q}=\mathcal{J}(\nabla_{x,\xi}\im q)=\mathcal{J}(-2 \J \im F)=2\im F.
$$ 
We deduce that the inclusion \eqref{we11} also reads as 
$$
\forall t > 0, \quad WF\big(e^{-tq^w(x,D)}u\big) \subseteq \left( \exp(tH_{\im q}) \left( WF(u) \cap S \right) \right) \cap S,
$$
where $H_{\im q}$ stands for the Hamilton vector field associated to the imaginary part of $q$.
\end{proof}

Finally we consider the particular case when the quadratic operator $q^w(x,D)$ commutes with its formal adjoint $\overline{q}^w(x,D)$ when acting on $\mathscr{S}(\rr d)$,
$$
q^w(x,D) \overline{q}^w(x,D) = \overline{q}^w(x,D) q^w(x,D) u, \quad u \in \cS (\rr d). 
$$
This condition of normality is equivalent to the condition on the Poisson bracket 
$$
\{q,\overline{q}\} 
= \la q_{\xi}', \overline q_{x}' \ra - \la q_{x}', \overline q_{\xi}' \ra 
= 2i\{\im q,\re q\}=0.
$$ 
The condition is also equivalent to $[\re F,\im F]=0,$
since the Hamilton map of the Poisson bracket $\{\im q,\re q\}$ is given by $-2[\im F,\re F]$, see e.g. \cite[Lemma~2]{Pravda-Starov1}.
In this case, we deduce from the definition \eqref{h1} that the singular space reduces to 
$$
S=\Ker (\re F)\cap T^*\rr d.
$$
We obtain the following corollary:

\medskip

\begin{cor}\label{co1}
Let $q$ be a quadratic form on $T^*\rr d$ defined by a symmetric matrix $Q \in \cc {2d \times 2d}$ satisfying $\re Q \geq 0$, and $F = \J Q$ its Hamilton map.
When the Poisson bracket $\{q,\overline{q}\}=0$ is identically equal to zero, then the following microlocal inclusion holds for all $u \in \cS'(\rr d)$, $t>0$,
\begin{equation*}
WF\big(e^{-tq^w(x,D)}u\big) 
\subseteq \left( \exp(tH_{\emph{\textrm{Im}}q})\big( WF(u) \cap \Ker (\re F) \big) \right)
 \cap \Ker (\re F),
\end{equation*}
where $\exp(tH_{\emph{\textrm{Im}}q})=e^{2t \im F}$ stands for the flow of 
the Hamilton vector field $H_{\emph{\textrm{Im}}q}$ associated to $\im q$.
\end{cor}

\section{Some particular equations}
\label{secparticulareq}

In this section, we look at propagation of the Gabor wave front set according to Theorem \ref{th1} and Corollary \ref{co1} in some particular cases of quadratic forms $q$. 

\subsection{A first example}

Let us consider the equation
\begin{equation}\label{heatvariant1}
\partial_t u(t,x) + \la x,A x \ra u(t,x) = 0, \quad t \geq 0, \quad x \in \rr d, 
\end{equation}
where $0 \leq A \in \rr {d \times d}$ is symmetric. 
It is a particular case of the general equation \eqref{schrodeq} where the quadratic form $q(x,\xi) = \la x, A x \ra$ 
is defined by the matrix
\begin{equation*}
Q = \left(
\begin{array}{cc}
A & 0 \\
0 & 0
\end{array}
\right),
\end{equation*}
and the Hamiltonian is
\begin{equation*}
F = \J Q = \left(
\begin{array}{cc}
0 & 0 \\
-A & 0
\end{array}
\right). 
\end{equation*}
The solution operator is 
\begin{equation*}
e^{-t q^w(x,D)} u (x) 
= e^{- t \la x, A x \ra} u(x). 
\end{equation*}
We have $\Ker (\re F) = \Ker A \times \rr d$ so Corollary \ref{co1} gives
\begin{equation}\label{propagation1}
WF( e^{- t \la \cdot, A \cdot \ra}  u) 
\subseteq  WF(u) \cap ( \Ker A \times \rr d ), \quad t > 0, \quad u \in \cS'(\rr d). 
\end{equation}
This example gives a generalization of \eqref{example3} as follows. 
Let $t=1$ and $u=1$. By \eqref{example2} and \eqref{propagation1} we have
\begin{equation}\label{propagation1a}
WF( e^{- \la \cdot, A \cdot \ra} ) 
\subseteq \Ker A \setminus 0 \times \{ 0 \}. 
\end{equation}
To see the opposite inclusion, let $x \in \Ker A \setminus 0$ and let $\fy(y) = e^{-|y|^2}$ for $y \in \rr d$. 
The STFT evaluated at $(t x, 0) \in T^* \rr d$ is then for any $t>0$
\begin{align*}
V_\fy (e^{- \la \cdot, A \cdot \ra} )( t x, 0)
& = \int_{\rr d} e^{- \la y, A y \ra - |y-tx|^2 } \, dy \\
& = \int_{\rr d} e^{- \la tx+ y, A (tx+y) \ra - |y|^2  } \, dy \\
& = \int_{\rr d} e^{- \la y, A y \ra - |y|^2 } \, dy, 
\end{align*}
which is a positive constant that does not depend on $t>0$. 
The STFT cannot therefore decay superpolynomially in any cone in $T^* \rr d$ containing $(t x, 0)$. 
This proves the opposite inclusion to \eqref{propagation1a} so we have
\begin{equation}\label{propagation1b}
WF( e^{- \la \cdot, A \cdot \ra} ) 
= \Ker A \setminus 0 \times \{ 0 \} \subseteq T^* \rr d. 
\end{equation}
Let $A \in \cc {d \times d}$ be symmetric with $\im A \geq 0$. Since $e^{i \la x, \re A x \ra/2 } = \mu(\chi)$
where 
\begin{equation*}
\chi = 
\left(
\begin{array}{ll}
I_d & 0 \\
\re A  & I_d 
\end{array}
\right) \in \Sp(d,\ro),
\end{equation*}
(cf. \eqref{symplecticoperator} and \eqref{chidef}),  
we obtain from \eqref{symplecticinvarianceWF} and \eqref{propagation1b}
\begin{equation*}
\begin{aligned}
WF(  e^{i \la x, A x \ra/2 }  ) 
& = WF(  e^{i \la x, \re A x \ra/2 } e^{- \la x, \im A x \ra/2 } ) \\
& = \chi WF ( e^{- \la x, \im A x \ra/2 } ) \\
& = \{ (x, (\re A) x + \xi):  \, (x,\xi)  \in WF( e^{- \la x, \im A x \ra/2 } )  \} \\
& = \{ (x, (\re A) x ):  \, x \in \rr d \cap \Ker (\im A) \setminus 0 \},
\end{aligned}
\end{equation*}
which is the announced generalization of \eqref{example3}.

\subsection{The heat equation}

The heat equation 
\begin{equation}\label{heatequation}
\partial_t u(t,x) - \Delta_x u(t,x) = 0, \quad t \geq 0, \quad x \in \rr d, 
\end{equation}
is a particular case of the general equation \eqref{schrodeq} where $q(x,\xi) = |\xi|^2$. 
Thus $q$ is defined by the matrix
\begin{equation*}
Q = \left(
\begin{array}{cc}
0 & 0 \\
0 & I_d
\end{array}
\right) \in \cc {2d \times 2d},
\end{equation*}
and the Hamiltonian is
\begin{equation*}
F = \J Q = \left(
\begin{array}{cc}
0 & I_d \\
0 & 0
\end{array}
\right). 
\end{equation*}
Therefore $\Ker (\re F) = \rr d \times \{0\}$ so Corollary \ref{co1} yields
\begin{equation}\label{heatinclusion}
WF(e^{- t q^w(x,D)} u) 
\subseteq  WF(u) \cap ( \rr d \times \{ 0 \} ), \quad t > 0, \ u \in \cS'(\rr d). 
\end{equation}
We may conclude that the solution to the heat equation is immediately regularizing
\begin{equation*}
e^{-t q^w(x,D)} u \in \cS(\rr d), \quad t > 0, 
\end{equation*}
provided the Gabor wave front set of the initial data $u \in \cS'(\rr d)$ satisfies
\begin{equation*}
WF(u) \cap ( \rr d \times \{ 0 \} ) = \emptyset. 
\end{equation*}
Note that (cf. \eqref{example2})
\begin{equation*}
(\rr d \setminus \{ 0 \}) \times \{ 0 \} = WF( e^{i \la \cdot, \xi \ra} ), 
\end{equation*}
for any $\xi \in \rr d$, in particular $e^{i \la \cdot, \xi \ra}=1$ when $\xi = 0$. 
If the initial data is $u=1$, then $u(t,\cdot)=1 \notin \cS$ for any $t \geq 0$ so the equation is not regularizing. 
The inclusion \eqref{heatinclusion} says that the equation is immediately regularizing if $WF(u)$ is disjoint from $WF(1)$. 

\subsection{Another equation with real $q$}

Next we look at the equation
\begin{equation}\label{heatvariant2}
\partial_t u(t,x) + (|x|^2 - \Delta_x) u(t,x) = 0, \quad t \geq 0, \quad x \in \rr d, 
\end{equation}
for which $q(x,\xi) = | x |^2 + |\xi|^2$ 
is defined by the matrix
$Q = I_{2d}$, 
and the Hamiltonian is
$F = \J Q = \J$. 
Thus $\Ker (\re F) = \{ (0,0) \}$ so Corollary \ref{co1} yields
$WF(e^{-t q^w(x,D)} u) = \emptyset$ for $u \in \cS'(\rr d)$ and $t>0$.
Thus we have immediate regularization
\begin{equation}\label{propagation5}
e^{-t q^w(x,D)} u \in \cS(\rr d), \quad t>0, \ u \in \cS'(\rr d). 
\end{equation}

\subsection{The free Schr\"odinger equation}

The free Schr\"odinger equation 
\begin{equation*}
\partial_t u(t,x) - i \Delta_x u(t,x) = 0, \quad t \in \ro, \quad x \in \rr d, 
\end{equation*}
is a particular case of the general equation \eqref{schrodeq} where $q(x,\xi) = i|\xi|^2$. 
Thus $q$ is purely imaginary-valued and defined by the matrix
\begin{equation*}
Q = \left(
\begin{array}{cc}
0 & 0 \\
0 & iI_d
\end{array}
\right),
\end{equation*}
and the Hamiltonian is
\begin{equation*}
F = \J Q = \left(
\begin{array}{cc}
0 &  iI_d \\
0 & 0
\end{array}
\right). 
\end{equation*}
We have
\begin{equation*}
e^{-2 i t F} = I_{2d} - 2 i t F
= \left(
\begin{array}{cc}
I_d & 2 t I_d \\
0 & I_d
\end{array}
\right).
\end{equation*}
Since $q$ is purely imaginary-valued the propagation of the Gabor wave front set is given exactly by \eqref{realcase} 
\begin{equation*}
WF(e^{-t q^w(x,D)} u) = 
\left( 
\begin{array}{cc}
I_d & 2 t I_d \\
0 & I_d
\end{array}
\right) WF(u), \quad t \in \ro, \ u \in \cS'(\rr d). 
\end{equation*}

\subsection{The harmonic oscillator}

The harmonic oscillator Schr\"odinger equation 
\begin{equation*}
\partial_t u(t,x) + i(|x|^2- \Delta_x) u(t,x) = 0, \quad t \in \ro, \quad x \in \rr d, 
\end{equation*}
is a particular case of the general equation \eqref{schrodeq} where $q(x,\xi) = i (|x|^2 + |\xi|^2)$. 
Thus $q$ is defined by the matrix $Q =i  I_{2d}$
and the Hamiltonian is $F = \J Q = i \J$.
We have (cf. \cite[p.~188]{Folland1})
\begin{equation*}
e^{-2 i t F} 
= e^{2 t \J}
= \left(
\begin{array}{cc}
\cos(2t) I_d & \sin(2t) I_d \\
-\sin(2t) I_d & \cos(2t) I_d \\
\end{array}
\right).
\end{equation*}
Since $q$ is again purely imaginary-valued the propagation of the Gabor wave front set is periodic, given exactly by \eqref{realcase} for $u \in \cS'(\rr d)$
\begin{equation*}
WF(e^{- t q^w(x,D)} u) = 
\left(
\begin{array}{cc}
\cos(2t) I_d & \sin(2t) I_d \\
-\sin(2t) I_d & \cos(2t) I_d \\
\end{array}
\right) WF(u), \quad t \in \ro. 
\end{equation*}

\subsection{Equations with nonzero real and imaginary parts of $q$}

Next we study the equation
\begin{equation}\label{schrodvariant1}
\partial_t u(t,x) - (1+i)\Delta_x u(t,x) = 0, \quad t \geq 0, \quad x \in \rr d, 
\end{equation}
for which $q(x,\xi) = (1+i) |\xi|^2$ 
is defined by the matrix
\begin{equation*}
Q = \left(
\begin{array}{cc}
0 & 0 \\
0 & (1+i) I_d
\end{array}
\right),
\end{equation*}
and the Hamiltonian is
\begin{equation*}
F = \left(
\begin{array}{cc}
0 & (1+i) I_d \\
0 & 0
\end{array}
\right). 
\end{equation*}
Since $[\re F, \im F]=0$, which is equivalent to $\{q, \overline q \}=0$, we may apply Corollary \ref{co1}. 
We have $\Ker (\re F) = \rr d \times \{ 0 \}$ and
\begin{equation*}
e^{2 t \im F} = I_{2d} + 2 t \im F
= \left(
\begin{array}{cc}
I_d & 2 t I_d \\
0 & I_d
\end{array}
\right). 
\end{equation*}
Corollary \ref{co1} gives for $u \in \cS'(\rr d)$
\begin{equation*}
\begin{aligned}
WF(e^{-t q^w(x,D)} u) 
& \subseteq  
\left(
\begin{array}{cc}
I_d & 2 t I_d \\
0 & I_d
\end{array}
\right)
\left( WF(u)  \cap  (   \rr d \times \{ 0 \}) \right) \\
& =  
WF(u)  \cap  \left(   (\rr d \setminus \{ 0 \} ) \times \{ 0 \}  \right), \quad t > 0. 
\end{aligned}
\end{equation*}
The conclusion is thus the same as for the heat equation. 

Let us next study the equation
\begin{equation}\label{schrodvariant2}
\partial_t u(t,x) +(|x|^2 - i\Delta_x) u(t,x) = 0, \quad t \geq 0, \quad x \in \rr d, 
\end{equation}
for which $q(x,\xi) = |x|^2 + i |\xi|^2$ 
is defined by the matrix
\begin{equation*}
Q = \left(
\begin{array}{cc}
I_d & 0 \\
0 & i I_d
\end{array}
\right),
\end{equation*}
and the Hamiltonian is
\begin{equation*}
F = \left(
\begin{array}{cc}
0 & i I_d \\
-I_d & 0
\end{array}
\right). 
\end{equation*}
Since $[\re F, \im F] \neq 0$ we have $\{q, \overline q \} \neq 0$ so we cannot apply Corollary \ref{co1}. 
But Theorem \ref{th1} is applicable. 
We compute the singular space $S$. 
We have $\Ker (\re F) = \{ 0 \} \times \rr d$ and $\Ker (\re F \, \im F) = \rr d \times \{ 0 \}$,
which taken together imply $S=\{ (0,0) \}$. 
By Theorem \ref{th1},
\begin{equation*}
e^{-t q^w(x,D)} u \in \cS(\rr d), \quad t > 0, \ u \in \cS'(\rr d). 
\end{equation*}

\subsection{Non-selfadjoint equations with very degenerate diffusions}

We study first the equation
$$
\partial_tu+ x_1^2u+2\sum_{j=1}^{d-1}x_{j+1}\partial_{x_j}u-i \Delta_x u=0, \quad t \geq 0, \quad x \in \rr d, \quad d \geq 2,
$$
for which 
$$
q(x,\xi)=x_1^2+i\Big(\sum_{j=1}^d\xi_j^2+2\sum_{j=1}^{d-1}x_{j+1}\xi_j\Big).
$$
It was checked in \cite[p.~48--49]{kps21} that the Hamilton map of $q$ is given by
$$
(y_1,\dots,y_d,\eta_1,\dots,\eta_d)=F(x_1,\dots,x_d,\xi_1,\dots,\xi_d),
$$ 
with $y_j = i (\xi_j+x_{j+1})$ for $1 \leq j \leq d-1$, $y_{d} = i \xi_{d}$,
$\eta_1=-x_1$, $\eta_j=-i\xi_{j-1}$ for $2 \leq j \leq d$, and that the singular space is zero:
$$
S=\Big(\bigcap_{j=0}^{2d-1} \Ker \big[\re F(\im F)^j \big]\Big) \cap T^*\rr d =\{0\}.
$$
Starting with an initial datum given by any tempered distribution $u \in \cS'(\rr d)$, there is therefore a Schwartz smoothing for any positive time
$$
e^{-tq^w(x,D)}u \in \cS (\rr d), \quad t>0.
$$
Finally we consider the equation
$$
\partial_t u+ x_1^2u+2\sum_{j=1}^{d-2}x_{j+1}\partial_{x_j}u-i\Delta_x u+i x_d^2 u=0, \quad t \geq 0, \quad x \in \rr d, \quad d \geq 3,
$$
for which 
$$
q(x,\xi) = x_1^2+i\Big(\sum_{j=1}^d\xi_j^2+2\sum_{j=1}^{d-2}x_{j+1}\xi_j+x_d^2\Big).
$$
The Hamilton map of $q$ is 
$$
(y_1,\dots,y_d,\eta_1,\dots,\eta_d) = F(x_1,\dots,x_d,\xi_1,\dots,\xi_d),
$$ 
with $y_j=i(\xi_j+x_{j+1})$ for $1 \leq j \leq d-2$, $y_{d-1}=i\xi_{d-1}$, $y_d=i\xi_d$,
$\eta_1=-x_1$, $\eta_j=-i\xi_{j-1}$ for $2 \leq j \leq d-1$, $\eta_d=-ix_d$. The singular space is
$$
S=\Big(\bigcap_{j=0}^{2d-1} \Ker \big[\re F(\im F)^j \big]\Big) \cap T^*\rr d=(\{0\}_{d-1} \times \ro) \times (\{0\}_{d-1} \times \ro),
$$
with $\{0\}_{d-1}=(0,\dots,0) \in \rr {d-1}$. 
Theorem \ref{th1} shows that for all $u \in \cS'(\rr d)$, $t>0$,  
$$
WF(e^{-t q^w(x,D)}u) \subseteq \pi_1\left(\{0\}_{2d-2} \times  
\left(\begin{array}{cc}
  \cos(2t) & \sin(2t) \\
  -\sin(2t) & \cos(2t)
 \end{array}\right)
\pi_0(WF(u))\right),
$$
where the mappings $\pi_0 : \rr {2d} \rightarrow \rr  2$ and $\pi_1 : \rr {2d} \rightarrow \rr {2d}$ are defined by 
\begin{align*}
\pi_0(x_1,\dots,x_d,\xi_1,\dots,\xi_d) & = (x_d,\xi_d), \\ 
\pi_1(x_1,\xi_1,x_2,\xi_2,\dots,x_d,\xi_d) & = (x_1,\dots,x_d,\xi_1,\dots,\xi_d).
\end{align*}

\section{Examples and counterexamples}
\label{secexamples}

In this final section we collect some observations on the Gabor wave front set of some basic tempered distributions. 
As a byproduct we find an example that shows that the heat propagator microlocal inclusion \eqref{heatinclusion} may be strict for some initial data $u \in \cS'(\rr d)$. For simplicity we work with $d=1$. 

First we study the family of functions 
\begin{equation*}
u(x) = e^{i x^p}, \quad x \in \ro, 
\end{equation*}
for an exponent $p>0$. If $x<0$ then 
\begin{equation*}
u(x) = \exp(i (-1)^p |x|^p) = \exp(i \cos(\pi p) |x|^p - \sin(\pi p) |x|^p),
\end{equation*}
so $u \in \cS'(\ro)$ if and only if there exists $n \in \no_0 = \no \cup \{0\}$ such that $2n \leq p \leq 2n+1$. 
We need the following lemma which is proved by a straightforward induction argument. 

\begin{lem}\label{Tlemma}
If $\xi \in \ro$, $p>1$, $\fy \in C^\infty(\ro)$ and the operator $T$ is defined by
\begin{equation*}
T \fy (y) = \frac{\partial}{\partial y} \left( \frac{\fy(y)}{p y^{p-1} - \xi} \right), \quad y \in \ro, 
\end{equation*}
for $\xi \neq p y^{p-1}$, then for $m \in \no$ we have 
\begin{equation}\label{generalcase}
T^m \fy (y) = \sum_{n \geq m, \, \ell \geq 0, \, s > -m: \, (p-1)n - s \geq(p-1)m} C_{s,\ell,n,p} \frac{y^s \fy^{(\ell)}(y)}{(p y^{p-1} - \xi)^n},
\end{equation}
where the sum is finite and $C_{s,\ell,n,p}$ are real constants. 
If $2 \leq p \in \no$ then
\begin{equation}\label{integercase}
T^m \fy (y) = \sum_{n \geq m, \, \ell \geq 0, \, \no \ni s \geq 0: \, (p-1)n - s \geq(p-1)m} C_{s,\ell,n,p} \frac{y^s \fy^{(\ell)}(y)}{(p y^{p-1} - \xi)^n}.
\end{equation}
\end{lem}

\begin{prop}\label{frequencyaxis}
If $u(x) = e^{i x^p}$, $2n \leq p \leq 2n+1$ for some $n \in \no$, and $p>2$ then 
\begin{equation*}
WF(u) \subseteq \{ 0 \} \times (\ro \setminus 0).
\end{equation*}
\end{prop}

\begin{proof}
If $0 \neq z \notin \{ 0 \} \times (\ro \setminus 0)$ then there exists $C>0$ such that $z \in \Gamma = \{(x,\xi): \, |x| > C |\xi| \} \subseteq T^* \ro$ which is an open conic set. 
If $|y-x| \leq 1$ and $(x,\xi) \in \Gamma$ we have
\begin{equation}\label{estbelow1}
\begin{aligned}
|p y^{p-1} - \xi| & \geq p |y|^{p-1} - |\xi|
\geq p (|x|-1)^{p-1} - |\xi| \\
& = p |x|^{p-1} \left( (1-|x|^{-1})^{p-1} - \frac{|\xi|}{p |x|^{p-1}} \right) \\
& > p |x|^{p-1} \left( (1-|x|^{-1})^{p-1} - \frac{1}{p C |x|^{p-2}} \right) \\
& \geq \frac{p}{2} |x|^{p-1},
\end{aligned}
\end{equation}
provided $|x| \geq A$ for $A \geq 2$ sufficiently large. 
If $|x| < A$ and $(x,\xi) \in \Gamma$ then $|(x,\xi)|$, and therefore $ \eabs{(x,\xi)}^N |V_\fy u(x,\xi)|$, is uniformly bounded.
Hence it suffices to consider $|x| \geq A$ and $(x,\xi) \in \Gamma$. 
Let $\fy \in C_c^\infty(\ro)$ be real-valued with $\supp \fy \subseteq [-1,1]$ and let $N \geq 0$. 
Since $A \geq 2$ we have $|y| \geq 1$ when $|y-x| \leq 1$. 
Due to
\begin{equation*}
\frac{\partial}{\partial y} e^{i (y^p - y \xi)} = i(p y^{p-1} - \xi) e^{i (y^p - y \xi)}, 
\end{equation*}
we obtain by integration by parts for any $m \in \no$ using  \eqref{generalcase} in Lemma \ref{Tlemma}
if $(x,\xi) \in \Gamma$
\begin{equation}\label{stftest1}
\begin{aligned}
& \eabs{(x,\xi)}^N |V_\fy u(x,\xi)| \\
& \lesssim \eabs{x}^N \left| \int_\ro e^{i (y^p - y \xi)} \fy(y-x) dy \right| \\
& = \eabs{x}^N \left| \int_{|y-x| \leq 1} e^{i (y^p - y \xi)} T^m \fy(y-x) dy \right| \\
& \lesssim \eabs{x}^N \sum_{n \geq m, \, \ell \geq 0, \, s > -m: \, (p-1)n -s \geq(p-1)m} \int_{|y-x| \leq 1}  
\frac{|y|^s | \fy^{(\ell)}(y-x)|}{|p y^{p-1} - \xi|^n} dy. 
\end{aligned}
\end{equation}
For $s<0$ we can estimate $|y|^s \leq 1$ since $|y| \geq 1$ when $|y-x| \leq 1$. 
In this case we may estimate using \eqref{estbelow1}, $|x| \geq A$ and $n \geq m$,
\begin{equation}\label{sneg}
\frac{|y|^s}{|p y^{p-1} - \xi|^n} \lesssim \eabs{x}^{-n(p-1)} \leq \eabs{x}^{-m(p-1)}. 
\end{equation} 
If instead $s \geq 0$ we estimate $|y|^s \leq \eabs{y}^s \lesssim \eabs{x}^s \eabs{y-x}^s$
and again using \eqref{estbelow1}, $|x| \geq A$ and $(p-1)n -s \geq(p-1)m$,
\begin{equation}\label{spos}
\frac{|y|^s}{|p y^{p-1} - \xi|^n} \lesssim \eabs{x}^{s-n(p-1)} \eabs{y-x}^s \leq \eabs{x}^{-m(p-1)} \eabs{y-x}^s. 
\end{equation} 
Inserting \eqref{sneg} and \eqref{spos} into \eqref{stftest1} gives if $(x,\xi) \in \Gamma$ and $|x| \geq A$,
\begin{equation*}
\eabs{(x,\xi)}^N |V_\fy u(x,\xi)| 
\lesssim \eabs{x}^{N-(p-1)m} \leq 1,
\end{equation*}
provided $m$ is sufficiently large. 
Thus $z \notin WF(u)$. 
\end{proof}

If $u(-x) = \pm u(x)$ then $|V_\fy u(x,\xi)| = |V_{\check \fy} u(-x,-\xi)|$ where $\check \fy(x) = \fy(-x)$, so $\check u = \pm u$ implies
\begin{equation}\label{mirror}
WF(u) = - WF(u). 
\end{equation}
Since $u(x) = e^{i x^p} \notin \cS (\ro)$ this gives

\begin{cor}
If $u(x) = e^{i x^{2k}}$ with $k \in \no$ and $k \geq 2$ then 
\begin{equation*}
WF(u) = \{ 0 \} \times (\ro \setminus 0).
\end{equation*}
\end{cor}

In the case when $p \geq 3$ is an odd integer the Gabor wave front set is described by the following result where we denote $\ro_+ = \{ x \in \ro: \, x>0 \}$. 

\begin{prop}\label{WFairy}
If $u(x) = e^{i x^{2k+1}}$ with $k \in \no$ and $k \geq 1$ then 
\begin{equation*}
WF(u) = \{ 0 \} \times \ro_+.
\end{equation*}
\end{prop}

\begin{proof}
Since $u \notin \cS$ we must have $WF(u) \neq \emptyset$, and hence 
by Proposition \ref{frequencyaxis} it suffices to show $(0,-1) \notin WF(u)$. 
The open conic set $\Gamma = \{(x,\xi): \, \xi < - |x| \} \subseteq T^* \ro$ contains  $(0,-1)$. 
Let $\fy \in C_c^\infty(\ro)$ be real-valued with $\supp \fy \subseteq [-1,1]$ and let $N \geq 0$. Since $\eabs{(x,\xi)} \lesssim \eabs{\xi}$ when $(x,\xi) \in \Gamma$ it suffices to show 
\begin{equation}\label{reducedestimate}
\sup_{(x,\xi) \in \Gamma, \, |\xi| \geq 1} | \xi^N V_\fy u(x,\xi)| < +\infty. 
\end{equation}
If $(x,\xi) \in \Gamma$ then 
\begin{equation}\label{estbelow2}
|(2k+1) y^{2k} - \xi| = (2k+1) y^{2k} + |\xi| \geq \max( (2k+1) y^{2k}, | \xi |). 
\end{equation}
Integration by parts gives
\begin{equation*}
\begin{aligned}
| \xi^N V_\fy u(x,\xi)| 
& = \left| \int_\ro (i \partial_y)^N \left( e^{- iy \xi} \right) e^{i y^{2k+1}} \fy(y-x) dy \right| \\
& = \left| \sum_{n \leq N} \int_{|y-x| \leq 1} e^{i (y^{2k+1} - y \xi)} p_n (y)  \fy^{(n)} (y-x) dy \right|, \\
\end{aligned}
\end{equation*}
where $p_n$ are polynomials of degree $\leq 2 k N$. 
From \eqref{integercase} in Lemma \ref{Tlemma} with $p=2k+1$ we have for any $m \in \no$,
\begin{align}\label{integralestimate1}
& \ \left| \int_{|y-x| \leq 1} e^{i (y^{2k+1} - y \xi)} p_n (y) \fy^{(n)} (y-x) dy \right| \\ \notag
 = & \ \left| \int_{|y-x| \leq 1} e^{i (y^{2k+1} - y \xi)} T^m \left( p_n (y)  \fy^{(n)} (y-x) \right) dy \right| \\  \notag
 \lesssim & \ \sum_{\substack{n \geq m, \, \ell \geq 0, \, \no \ni s \geq 0 \\ 2k n - s \geq 2 k m}} \int_{|y-x| \leq 1} \frac{\eabs{y}^{s} }{|(2k+1) y^{2k} - \xi |^n} \left| \partial_y^\ell \left( p_n (y) \fy^{(n)}(y-x) \right) \right| dy.
\end{align}
Using \eqref{estbelow2} and the assumption $|\xi| \geq 1$ we estimate the integrand when $|y| < 1$ as
\begin{equation}\label{sneg2}
\frac{\eabs{y}^s}{|(2k+1) y^{2k} - \xi|^n} \left| \partial_y^\ell \left( p_n (y) \fy^{(n)}(y-x) \right) \right|
\lesssim  |\xi|^{-n} \leq 1. 
\end{equation} 
When $|y| \geq 1$ we estimate the integrand as, again using \eqref{estbelow2}, plus $\deg p_n \leq 2 k N$ and $2k n - s \geq 2 k m$, 
\begin{multline}\label{spos2}
\frac{\eabs{y}^s}{|(2k+1) y^{2k} - \xi|^n} \left| \partial_y^\ell \left( p_n (y) \fy^{(n)}(y-x) \right) \right|\\
\lesssim \eabs{y}^{s-2kn+2kN}  \leq \eabs{y}^{2k(N-m)} \leq 1, 
\end{multline} 
if $m \geq N$. 
Inserting \eqref{sneg2} and \eqref{spos2} into \eqref{integralestimate1} shows that the latter integral is bounded when $(x,\xi) \in \Gamma$ and $|\xi| \geq 1$. 
This shows that \eqref{reducedestimate} is valid. 
\end{proof}

\begin{example}
Consider the heat propagator inclusion \eqref{heatinclusion} for $d=1$: 
\begin{equation}\label{heatinclusion1}
WF(e^{-t q^w(x,D)} u) 
\subseteq  WF(u) \cap ( \ro \times \{ 0 \} ), \quad t > 0, \ u \in \cS'(\ro). 
\end{equation}
Let 
\begin{equation}\label{airy}
u = \cF^{-1} (e^{i x^3}),
\end{equation}
which is (up to normalization) the Airy function (cf. \cite[Chapter 7.6]{Hormander0}). 
The heat equation \eqref{heatequation} for $d=1$ has the solution operator
$$e^{- t q^w(x,D)} u(x) = (2 \pi)^{-1} \int_{\ro} e^{i x \xi-t |\xi|^2} \widehat u (\xi) \, d \xi = \cF^{-1} ( e^{-t |\cdot|^2} \widehat u )(x), \quad x \in \ro.$$
If $t>0$ then
\begin{equation*}
e^{-t |\xi|^2} \widehat u(\xi) = e^{-t |\xi|^2 + i \xi^3} \in \cS(\ro),
\end{equation*}
so we have $WF(e^{-t q^w(x,D)} u) = \emptyset$ for $t>0$, that is, the equation is immediately regularizing for the initial datum \eqref{airy}. 

From Proposition \ref{WFairy} and $\J WF(u) = WF (\widehat u)$ (cf \eqref{symplecticinvarianceWF}) we obtain
\begin{equation*}
WF(u) = -\J WF (\widehat u) = \ro_{-} \times \{ 0 \},
\end{equation*}
where $\ro_- = \{x \in \ro: \, x<0 \}$. 
Thus 
\begin{equation*}
WF(u) \cap ( \ro \times \{ 0 \} ) = \ro_{-} \times \{ 0 \}, 
\end{equation*}
and hence this example shows that the inclusion \eqref{heatinclusion1} may be strict. 
\end{example}

The Gabor wave front set of the derivative of order $k \in \no_0$ of the Dirac delta on $\ro$ satisfies by microlocality 
\begin{equation*}
WF(\partial^k \delta_0) \subseteq WF(\delta_0) = \{0\} \times \ro \setminus 0. 
\end{equation*}
$WF(\partial^k \delta_0)$ satisfies \eqref{mirror} since $u = \partial^k \delta_0$ satisfies $\check u = (-1)^k u$.
Since $WF(\partial^k \delta_0)$ must be non-empty we conclude
\begin{equation*}
WF(\partial^k \delta_0) =  \{0\} \times \ro \setminus 0, \quad k \in \no_0. 
\end{equation*}
Combined with $WF (\widehat u) = \J WF(u)$ this gives the Gabor wave front set of a monomial
\begin{equation*}
WF(x^k) = \ro \setminus 0 \times  \{0\}, \quad k \in \no_0. 
\end{equation*}
Finally we determine the Gabor wave front set of the Heaviside function $H$.  

\begin{prop}
\begin{equation*}
WF(H) = \{ 0 \} \times (\ro \setminus 0) \cup \ro_+ \times \{ 0 \}. 
\end{equation*}
\end{prop}

\begin{proof}
Since $\delta_0 = H'$, we have by microlocality
\begin{equation*}
\{ 0 \} \times (\ro \setminus 0) = WF(\delta_0) \subseteq WF(H). 
\end{equation*}
If $x \neq 0$ and $\xi \neq 0$ then for some $C_1,C_2>0$, $(x,\xi)$ belongs to the open conic set
\begin{equation*}
\Gamma = \{ (x,\xi) : \, C_1 |x| < |\xi| < C_2 |x| \}. 
\end{equation*}
Let $\fy \in C_c^\infty(\ro)$ be real-valued with $\supp \fy \subseteq [-1,1]$ and let $N \geq 0$.
For $|x|>1$ we have for any $N \in \no$ integrating by parts
\begin{equation*}
\begin{aligned}
\left| \xi^N V_\fy H(x,\xi) \right|
& = \left| \int_0^\infty (i \partial_y)^N ( e^{- i y \xi} ) \fy(y-x) \, dy \right| \\
& = \left| \int_0^\infty e^{- i y \xi}   \partial^N \fy(y-x) \, dy \right|
\lesssim 1, 
\end{aligned}
\end{equation*}
and since $\eabs{(x,\xi)} \lesssim \eabs{\xi}$ when $(x,\xi) \in \Gamma$, 
and $\eabs{(x,\xi)} \lesssim 1$ when $|x| \leq 1$ and $(x,\xi) \in \Gamma$, it follows that 
\begin{equation*}
\{ 0 \} \times (\ro \setminus 0) \subseteq 
WF(H) \subseteq \{ 0 \} \times (\ro \setminus 0) \cup (\ro \setminus 0) \times \{ 0 \}. 
\end{equation*}
The open conic set $\Gamma=\{ (x,\xi): \, x < - |\xi|\}$ contains $(-1,0) \in T^* \ro$. 
If $|x| \leq 1$ then $|(x,\xi)| \lesssim 1$ when $(x,\xi) \in \Gamma$, and if 
$|x| > 1$ and $(x,\xi) \in \Gamma$ then $x<-1$ and 
\begin{equation*}
V_\fy H(x,\xi) = \int_0^\infty e^{- i y \xi}  \fy(y-x) \, dy = 0.
\end{equation*}
This shows that $(-1,0) \notin WF(H)$ so we have 
\begin{equation}\label{inclusions1}
\{ 0 \} \times (\ro \setminus 0) \subseteq 
WF(H) \subseteq \{ 0 \} \times (\ro \setminus 0) \cup \ro_+ \times \{ 0 \}. 
\end{equation}
Finally, suppose $\widehat \fy (0) \neq 0$. For $x>0$ fixed and $t>0$ we have 
\begin{equation*}
\left| V_\fy H( t x, 0) \right|
= \left| \int_{-t x}^\infty \fy(y) \, dy \right| 
\longrightarrow |\widehat \fy (0)|, \quad t \longrightarrow +\infty, 
\end{equation*}
which shows that $\ro_+ \times \{ 0 \} \subseteq WF(H)$. 
Together with \eqref{inclusions1} this proves the result. 
\end{proof}


\end{document}